\documentclass[reqno,11pt]{amsart}
\usepackage{amssymb}
\usepackage{mathrsfs}
\usepackage{dsfont}
\usepackage{amsmath,amssymb}
\usepackage{paralist}
\usepackage{graphics} 
\usepackage{epsfig} 
\usepackage[colorlinks=true]{hyperref}
\hypersetup{urlcolor=red, citecolor=blue}
\usepackage[top=1in,bottom=1in,left=1in,right=1in]{geometry}
\usepackage{cite}
\newtheorem{thm}{Theorem}[section]

\newtheorem{lem}[thm]{Lemma}

\theoremstyle{definition}
\newtheorem{defn}{Definition}[section]
\newtheorem{rem}{Remark}[section]
\numberwithin{equation}{section}
\newcommand{\eps}{\varepsilon}
\newcommand{\be}{\begin{equation}}
\newcommand{\ee}{\end{equation}}
\newcommand{\bes}{\begin{eqnarray}}
\newcommand{\ees}{\end{eqnarray}}
\newcommand{\bess}{\begin{eqnarray*}}
\newcommand{\eess}{\end{eqnarray*}}
\newcommand{\jfo}{\int_\Omega}

\newcommand{\jft}{\int^{t+1}_t\int_\Omega}
\newcommand{\zs}{\frac{m p}{3}+m+p-2}
\newcommand{\zbl}{(\cdot,t)}
\newcommand{\zbls}{(\cdot,s)}
\newcommand{\daR}{\mathbb{R}}

\newcommand{\sjl}{,\ \ \ \ for\ all\ t\geq0}
\newcommand{\sjjl}{,\ \ \ \ for\ all\ t>0}
\newcommand{\sj}{,\ \ \ \ {\rm for\ all}\ t\geq0}
\newcommand{\sjj}{,\ \ \ \ {\rm for\ all}\ t>0}
\newcommand{\banqun}[1]{{e^{-{#1}A}}}
\newcommand{\zzh}[1]{#1_\varepsilon}
\newcommand{\xrsl}{\stackrel{\ast}{\rightharpoonup}}
\newcommand{\lp}[1]{{L^{#1}(\Omega)}}

\begin{document}
\title[C.-S. system with slow $p$-Laplacian diffusion]{Global existence and boundedness in a chemotaxis-Stokes system with slow $p$-Laplacian diffusion}%
\author[Tao]{Tao Weirun}%
\address{Institute for Applied Mathematics, School of Mathematics, Southeast University, Nanjing 211189, P.R. China}
\email{taoweiruncn@163.com}
\author[Li]{Li Yuxiang}%
\address{Institute for Applied Mathematics, School of Mathematics, Southeast University, Nanjing 211189, P.R. China}
\email{lieyx@seu.edu.cn}


\begin{abstract}
This paper deals with a boundary-value problem in three-dimensional smooth bounded convex domains for the coupled chemotaxis-Stokes system with slow $p$-Laplacian diffusion
 \begin{eqnarray}\nonumber
  \left\{\begin{array}{lll}
     \medskip
     n_t+u\cdot\nabla n=\nabla\cdot(|\nabla n|^{p-2}\nabla n)-\nabla\cdot(n\nabla c),&{} x\in\Omega,\ t>0,\\
     \medskip
     c_t+u\cdot\nabla c=\Delta c-nc,&{} x\in\Omega,\ t>0,\\
     \medskip
     u_t=\Delta u+\nabla P+n\nabla\phi,&{} x\in\Omega,\ t>0,\\
     \medskip
     \nabla\cdot u=0,&{} x\in\Omega,\ t>0,
  \end{array}\right.
\end{eqnarray}
where $\phi\in W^{2,\infty}(\Omega)$ is the gravitational potential. It is proved that global bounded weak solutions exist whenever $p>\frac{23}{11}$ and
the initial data $(n_0,c_0,u_0)$ are sufficiently regular satisfying $n_0\geq 0$ and $c_0\geq 0$.
\end{abstract}
\maketitle
\section{Introduction}
In this paper, our goal is to establish the existence and the boundedness of the global weak solutions to the following chemotaxis-Stokes system with $p$-Laplacian diffusion
 \begin{eqnarray}\label{CS}
  \left\{\begin{array}{lll}
     \medskip
     n_t+u\cdot\nabla n=\nabla\cdot\left(|\nabla n|^{p-2}\nabla n\right)-\nabla\cdot(n\nabla c),&{} x\in\Omega,\ t>0,\\
     \medskip
     c_t+u\cdot\nabla c=\Delta c-nc,&{}x\in\Omega,\ t>0,\\
     \medskip
     u_t=\Delta u+\nabla P+n\nabla\phi ,&{}x\in\Omega,\ t>0,\\
     \medskip
     \nabla\cdot u=0, &{}x\in\Omega,\ t>0
  \end{array}\right.
\end{eqnarray}
in a smooth bounded domain $\Omega\subset\mathbb{R}^3$, where the scalar function $n$ represents the density of aerobic bacteria, and $c$ represents the concentration of oxygen. The vector $u =(u_1, u_2, u_3)$ is the fluid velocity field and $P=P(x, t)$ represents the associated pressure of the fluid. The given function $\phi$ stands for the gravitational potential produced by the action of physical forces on the cell. The nonlinear diffusion $\nabla\cdot(|\nabla n|^{p-2}\nabla n)$ is called the slow $p$-Laplacian diffusion if $p>2$, and called the fast $p$-Laplacian diffusion if $1 < p < 2$.

Chemotaxis describes biased movement of cells in response to the concentration gradient of a diffusible chemical signal. The most famous model used to describe this biochemotactic phenomenon is the Keller-Segel model which was first presented in \cite{KS1970}. The prototype of classical chemotaxis model reads as
 \begin{eqnarray}\label{idahuflia}
  \left\{\begin{array}{lll}
     \medskip
     n_t=\Delta n-\nabla\cdot(n\nabla c),&{} x\in\Omega,\ t>0,\\
     \medskip
     \tau c_t=\Delta c-c+n,&{}x\in\Omega,\ t>0,
  \end{array}\right.
\end{eqnarray}
where $n$ denotes the cell density and $c$ describes the concentration of the chemical signal which is directly produced by cells themselves. In the past 4 decades, this model has been attracted many mathematicians to study its qualitative properties such as critical mass phenomenon, blowup, boundedness, pattern formations and critical sensitivity exponents (e.g. see \cite{Calvez-CPDE-2012,Zhi-an-JMAA-2010,Hittmeir-SIAMJMA-2011,Horstmann&Wang-JAM-2001,Horstmann-Winkler-JDE-2005,MR1361006,Nagai&Senba&Yoshida-FE-1997,MR1887324,MR3124759,Winkler-JDE-2010,Winkler-JMPA-2013} and the references therein). In addition to the original model, a large number of variants of the classical form have also been studied, including the system with the logistic terms ( e.g. see \cite{Tao&Winkler-SIAMJMA-2011,Zhang&Li-ZAMP-2015-B,LiYan&Lankeit-N-2016}), multi-species chemotaxis system ( e.g. see \cite{Lou&Tao&Winkler-SIAMJMA-2014,Zhang&Li-JMAA-2014,LiYan&Li-NA-2014,LiYan-JMAA-2015}), attraction-repulsion chemotaxis system ( e.g. see\cite{Tao&Wang-M3AS-2013,LiYan&Li-NARWA-2016}) and so on ( e.g. see review articles \cite{Bellomo&Bellouquid&Tao&Winkler-M3AS-2015,Hillen-JMB-2009,Horstmann-JDMV-2003,Horstmann-JDMV-2004} for the further reading).

The chemotaxis-Navier-Stokes system was first proposed in \cite{Tuval2005}. Aerobic bacteria such as \emph{Bacillus subtilis} often live in thin fluid layers near solid-air-water contact line, in which the biology of chemotaxis, metabolism, and
cell-cell signaling is intimately connected to the physics of buoyancy,
diffusion, and mixing (cf. \cite{Tuval2005}). Both bacteria and oxygen diffuse through the fluid, and they are also transported by the fluid (cf. \cite{DOMBROWSKI-PRL-2004} and \cite{Lorz-M3AS-2010}). Taking all these roles into account, the model in \cite{Tuval2005} reads as
\begin{eqnarray}\label{model-Tuval}
  \left\{\begin{array}{lll}
     \medskip
     n_t+u\cdot\nabla n=\Delta n-\nabla\cdot(n\chi(c)\nabla c),&{} x\in\Omega,\ t>0,\\
     \medskip
     c_t+u\cdot\nabla c=\Delta c-nf(c),&{} x\in\Omega,\ t>0,\\
     \medskip
     u_t+\kappa(u\cdot\nabla) u=\Delta u+\nabla P+n\nabla\Phi,&{} x\in\Omega,\ t>0,\\
     \medskip
     \nabla\cdot u=0,&{} x\in\Omega,\ t>0,
  \end{array}\right.
\end{eqnarray}
where the domain $\Omega\subset \mathbb{R}^d$, the vector $u=(u_1(x,t), u_2(x,t), \cdots, u_d(x,t))$ is the fluid velocity field and the associated pressure is represented by $P=P(x,t)$.

In the last decade, the chemotaxis fluid system (\ref{model-Tuval}) has attracted much attention. In 2010, the author showed in \cite{Lorz-M3AS-2010} that in certain parameter regimes, the system (\ref{model-Tuval}) possesses local weak solutions in a bounded domain in $\daR^d$, $d=2,3$ with no-flux boundary condition and in $\daR^2$ in the case of inhomogeneous Dirichlet conditions for the oxygen. In the same year, the authors proved in \cite{Duan&Lorz&Markowich-CPDE-2010} that in the two-dimensional case, the Cauchy problem  for (\ref{model-Tuval}) with $\kappa=0$ admits global existence of weak solutions, provided that some further technical conditions are satisfied and the structural conditions on $\chi$ and $f$ are satisfied. It was also proved in \cite{Duan&Lorz&Markowich-CPDE-2010} that under the two-dimensional setting, the chemotaxis-Navier-Stokes system (\ref{model-Tuval}) with $\kappa=1$ and $\Omega=\daR^3$ admits global classical solutions near constant steady states. In 2011, global existence of solutions to the Cauchy problem was investigated under certain conditions in \cite{Liu&Lorz-AIHP-2011}. The authors showed there that the chemotaxis-Navier-Stokes system (the system (\ref{model-Tuval}) with $\kappa=1$) possesses global weak solutions for large data. In 2012, it was proved in \cite{Winkler-CPDE-2012} that if $\chi(s)\equiv1$ for $s\in\daR$ and $f(s)\equiv s$ for $s\in\daR$, then the simplified chemotaxis-Stokes system (the system (\ref{model-Tuval}) with $\kappa=0$) possesses at least one global weak solution and the full chemotaxis-Navier-Stokes system (the system (\ref{model-Tuval}) with any $\kappa\in\daR$) admits a unique global classical solution under the boundary condition $\frac{\partial n}{\partial\nu}=\frac{\partial c}{\partial\nu}=u=0$ on $\partial \Omega$ and suitable regularity assumptions on the initial data. In 2013, it was proved in \cite{Chae-DCDS-2013} that there exist a global classical solution to the Cauchy problem with $\Omega=\daR^2$ under the appropriate structural assumptions for $\chi$ and $f$. In 2014, the same author of \cite{Winkler-CPDE-2012} showed in \cite{Winkler-ARMA-2014} that the global classical solutions obtained in \cite{Winkler-CPDE-2012} stabilize to the spatially uniform equilibrium $(\bar{n}_0, 0, 0)$ with $\bar{n}_0=\frac1{|\Omega|}\int_\Omega n_0(x)dx$ as $t\rightarrow\infty$ if $\Omega$ is a bounded convex domain in $\daR^2$. In 2015, the authors proved in \cite{Zhang&Li-DCDS-2015} that such solution converges to the equilibrium $(\bar{n}_0, 0, 0)$ exponentially in time. In the same year, by deriving a new type of entropy-energy estimate, it was shown in \cite{Jiang&Wu&Zheng-2014} that the restricted condition in \cite{Winkler-ARMA-2014} that $\Omega$ is essentially assumed to be convex can be removed. In 2016, the author proved in \cite{LiYan&Li-JDE-2016} that global classical bounded solutions to the Cauchy problem exist for regular initial data. In the same year, the author in \cite{Winkler-AIHP-2016} established global weak solutions of (\ref{model-Tuval}) in bounded convex domains $\Omega\subset\daR^3$ with suitable regularity assumptions on the initial data and appropriate assumptions for $\chi$, $f$ and $\phi$. In 2017, the long-term behaviour of “eventual energy solution” was investigated in \cite{Winkler-2017-TAMS}, which, namely, become smooth on some interval $[T,\infty)$ and uniformly converge in the large-time-limit. For more results of the well-posedness of the Cauchy problem to (\ref{model-Tuval}) in the whole space we refer the reader to \cite{Chae-CPDE-2014,Liu&Lorz-AIHP-2011,ZhangQian-NARWA-2014,Zhang&Zheng-SIAM-2014}.

Recently, a number of papers studied the Keller-Segel system with the linear diffusion replaced by the nonlinear diffusion. In \cite{Di&Lorz&Markowich-DCDS-2010}, the authors introduced the model
\begin{eqnarray}\label{p-laplacian-ks}
  \left\{\begin{array}{lll}
     \medskip
     n_t+u\cdot\nabla n=\nabla\cdot\left(D(n)\nabla n\right)-\nabla\cdot(n\chi(c)\nabla c),&{} x\in\Omega,\ t>0,\\
     \medskip
     c_t+u\cdot\nabla c=\Delta c-nf(c),&{}x\in\Omega,\ t>0,\\
     \medskip
     u_t+\kappa(u\cdot\nabla)u=\Delta u+\nabla P+n\nabla\Phi ,&{}x\in\Omega,\ t>0,\\
     \medskip
     \nabla\cdot u=0, &{}x\in\Omega,\ t>0.
  \end{array}\right.
\end{eqnarray}
Under the assumption of porous media, taking $D(n)=n^{m-1}$, global-in-time solution to the chemotaxis-Stokes system was constructed in \cite{Di&Lorz&Markowich-DCDS-2010} for general initial data if $m\in(\frac32, 2]$, while the same result holds in three-dimensional setting under the constraint $m \in(\frac{7+\sqrt{217}}{12}, 2]$. Intuitively, the nonlinear diffusion $\Delta u^m$ for $m>1$ can prevent the occurrence of blow up. Under the three-dimensional setting. In \cite{Liu&Lorz-AIHP-2011}, global-in-time weak solutions to the Cauchy problem was obtained for $m=\frac43$ under the appropriate structural assumptions for $\chi$. It was shown in \cite{Winkler-CVPDE-2015} that (\ref{p-laplacian-ks}) admits a global bounded weak solutions to the chemotaxis-stokes system in
bounded convex domains if $m>\frac76$. This partially extended a precedent result which asserted
global solvability within the larger range $m>\frac87$ , but only in a class of weak solutions locally bounded in $\overline{\Omega}\times [0,\infty)$ (cf.\cite{Tao&Winkler-AIHP-2013}). It was proved in \cite{Zhang&Li-JDE-2015} that the model possesses at least one global weak solution under the condition that $m\geq\frac23$. Recently, it was shown in \cite{Winkler-JDE-2018} that the chemotaxis-Stokes system admits a global bounded weak solutions under the assumption that $m>\frac98$. Moreover, the obtained solutions are shown to approach the spatially homogeneous steady state ($\frac1{|\Omega|}\int_\Omega n_0,0,0$) in the large time limit. For smaller values of $m > 1$, up to now existence results are limited to classes of possibly unbounded solutions (cf.\cite{Duan&Xiang-IMRN-2014}).
\vskip 3mm
In addition to porous media diffusion, $p$-Laplacian diffusion has also been considered in the Keller-Segel system. In \cite{Cong&Liu-KRM-2016}, the authors studied the following $p$-Laplacian Keller-Segel model in dimension $d\geq3$:
 \begin{eqnarray}\label{ljgmodel}
  \left\{\begin{array}{lll}
     \medskip
     n_t=\nabla\cdot\left(|\nabla n|^{p-2}\nabla n\right)-\nabla\cdot(n\nabla c),&{} x\in\mathbb{R}^d,\ t>0,\\
     \medskip
     -\Delta c=n,&{}x\in\mathbb{R}^d,\ t>0,\\
     \medskip
     n(x,0)=n_0(x), &{}x\in\mathbb{R}^d,
  \end{array}\right.
\end{eqnarray}
and they proved the existence of a uniform in time $L^\infty$ bounded weak solution for system (\ref{ljgmodel}) with the supercritical
diffusion exponent $1<p<\frac{3d}{d+1}$ in the multi-dimensional space $\mathbb{R}^d$
under the condition that the $L^{\frac{d(3-p)}p}$ norm of initial data is smaller than a
universal constant. They also proved the local existence of weak solutions and a
blow-up criterion for general $L^1\cap L^\infty$ initial data. 

In \cite{first-paper}, the authors studied the following chemotaxis-Navier-Stokes system with slow $p$-Laplacian diffusion
\begin{eqnarray}
  \left\{\begin{array}{lll}
     \medskip
     n_t+u\cdot\nabla n=\nabla\cdot(|\nabla n|^{p-2}\nabla n)-\nabla\cdot(n\chi(c)\nabla c),&{} x\in\Omega,\ t>0,\\
     \medskip
     c_t+u\cdot\nabla c=\Delta c-nf(c),&{} x\in\Omega,\ t>0,\\
     \medskip
     u_t+(u\cdot\nabla) u=\Delta u+\nabla P+n\nabla\Phi,&{} x\in\Omega,\ t>0,\\
     \medskip
     \nabla\cdot u=0,&{} x\in\Omega,\ t>0.
  \end{array}\right.
\end{eqnarray}
It is proved that if $p>\frac{32}{15}$ and under appropriate structural assumptions on $f$ and $\chi$, for all sufficiently smooth initial data $(n_0,c_0,u_0)$ the model possesses at least one global weak solution.

In this paper, we shall consider the existence and the boundedness of the global weak solutions for the chemotaxis-Stokes system (\ref{CS}) with initial conditions
\be
n(x,0)=n_0(x),\ \ c(x,0)=c_0(x)\ \ {\rm and}\ \ u(x,0)=u_0(x),\ \ \ \ x\in\Omega\label{eq-i}
\ee
and the boundary conditions
\be
|\nabla n|^{p-2}\frac{\partial n}{\partial \nu}=\frac{\partial c}{\partial \nu}=0\ \ {\rm and}\ \ u=0\ \ \ \ {\rm{on}}\ \partial\Omega\label{eq-bc}
\ee
in a bounded convex domain $\Omega\subset\mathbb{R}^3$ with smooth boundary, where $\nu$ is the exterior unit normal vector on $\partial\Omega$. As for the initial data, we shall suppose that
\begin{eqnarray}\label{eq-ic}
  \left\{\begin{array}{lcl}
     \medskip
     n_{0}\in C^\omega(\overline{\Omega})\ \ {\rm for\ some}\ \omega>0\ {\rm with}\ n_0\geq0\ {\rm in}\ \Omega\ {\rm and}\ n_0\not\equiv0,\\
     \medskip
     c_0\in W^{1,\infty}(\Omega)\ {\rm is\ nonnegative},\\
     \medskip
     u_0\in D(A^\alpha)\ \ \ \ {\rm for\ some }\ \alpha\in(\frac34,1),
  \end{array}\right.
\end{eqnarray}
where $A=-\mathcal{P}\Delta$ denotes the Stokes operator in $L_{\sigma}^2(\Omega)=\left\{\varphi\in (L^2(\Omega))^3|\nabla\cdot\varphi=0\right\}$ with its domain given by $D(A)=W^{2,2}(\Omega)\cap W_0^{1,2}(\Omega)\cap L_{\sigma}^2(\Omega)$, and with $\mathcal{P}$ representing the Helmholtz projection on $L^2(\Omega)$ (\cite{Sohr-2001-book}).

Before giving the main result, let us first give the definition of weak solution.
\begin{defn}\label{defnaaaa}
Let
\bes
&&n\in L_{loc}^{1}(\overline{\Omega}\times[0, \infty)),\\
 &&c\in L^\infty_{loc}(\overline{\Omega}\times[0,\infty))\cap L_{loc}^1([0, \infty);W^{1,1}(\Omega))\ \ \ \ \ \ and\\
 && u\in L_{loc}^1([0, \infty);W^{1,1}(\Omega;\daR^3)),
\ees
be such that $n\geq0$ and $c\geq0$ in $\Omega\times(0, \infty)$, and
\be
|\nabla n|^{p-1},\ n|\nabla c|,\ \mbox{and}\ n|u|\ \mbox{belong to}\ L_{loc}^1(\overline{\Omega}\times[0, \infty)).
\ee
We call $(n, c, u)$ a {\it{global weak solution}} of the chemotaxis-Stokes system (\ref{CS}) with initial condition (\ref{eq-i}) and boundary condition (\ref{eq-bc}) if $\nabla\cdot u=0$ in the distribution sense, if
\be
-\int^\infty_0\int_\Omega n\varphi_t-\jfo n_0\varphi(\cdot,0)=-\int^\infty_0\int_\Omega |\nabla n|^{p-2}\nabla n\cdot\nabla\varphi+\int^\infty_0\int_\Omega n\nabla c\cdot\nabla\varphi+\int^\infty_0\int_\Omega nu\cdot\nabla\varphi,\label{eq1-weak-sol}
\ee
for all $\varphi\in C^\infty_0(\overline{\Omega}\times[0,\infty))$ fulfilling $\frac{\partial \varphi}{\partial \nu}=0$ on $\partial \Omega\times(0,\infty)$,
\be
-\int^\infty_0\int_\Omega c\varphi_t-\jfo c_0\varphi(\cdot,0)=-\int^\infty_0\int_\Omega \nabla c\cdot\nabla\varphi-\int^\infty_0\int_\Omega nc\varphi+\int^\infty_0\int_\Omega cu\cdot\nabla\varphi,\label{eq2-weak-sol}
\ee
for all $\varphi\in C^\infty_0(\overline{\Omega}\times[0,\infty))$, and
\be
-\int^\infty_0\int_\Omega u\cdot\varphi_t-\jfo u_0\cdot\varphi(\cdot,0)=-\int^\infty_0\int_\Omega \nabla u\cdot\nabla\varphi+\int^\infty_0\int_\Omega n\nabla\phi\cdot\varphi,\label{eq3-weak-sol}
\ee
for all $\varphi\in C^\infty_0(\Omega\times[0,\infty);\daR^3)$
such that $\nabla\cdot\varphi\equiv0$ in $\Omega\times(0,\infty)$.
\end{defn}

Our main result reads as:
\begin{thm}\label{thm1.1}
Let $\Omega\subset\mathbb{R}^3$ be a bounded convex domain with smooth boundary and $\phi\in W^{2,\infty}(\Omega)$, and
\be
p>\frac{23}{11}.\label{eq0-thm1.1}
\ee
Then for each $n_0,c_0$ and $u_0$ satisfying (\ref{eq-ic}) there exist functions
\begin{eqnarray}\label{eq1-thm1.1}
  \left\{\begin{array}{lcl}
     \medskip
     n\in L^\infty(\Omega\times(0,\infty)),\\
     \medskip
     c\in \bigcap_{s>1}L^\infty((0,\infty);W^{1,s}(\Omega))\cap C^0(\overline{\Omega}\times[0,\infty))\cap C^{1,0}(\overline{\Omega}\times(0,\infty)),\\
     \medskip
     u\in  L^\infty(\Omega\times(0,\infty))\cap L^2_{loc}([0,\infty);W^{1,2}_0(\Omega)\cap L^2_\sigma(\Omega))\cap C^0(\overline\Omega\times[0,\infty))
  \end{array}\right.
\end{eqnarray}
such that the triple $(n, c, u)$ forms a global weak solution of (\ref{CS}), (\ref{eq-i}) and (\ref{eq-bc}) in the sense of Definition \ref{defnaaaa}.

\end{thm}
The rest of this paper is organized as follows. In Section 2, we introduce a family of regularized problems and give some preliminary properties. Based on an energy-type inequality, a priori estimates are given in Section 3. Section 4 to section 7 are several useful uniform estimates. Finally, we give the proof of the main result in Section 8.
\section{Regularized problems}
The global weak solution of (\ref{CS}) is constructed as the limit of appropriate regularized problems. We regularize the cross-diffusive term in (\ref{CS}) by introducing a family $(\chi_\varepsilon)_{\varepsilon\in(0,1)}\subset C^\infty_0([0,\infty))$ (cf.\cite{Winkler-JDE-2018}) fulfilling
\be\label{regular-chi}
0\leq\chi_\varepsilon\leq1\ {\rm in}\ [0,\infty),\ \ \ \ \chi_\varepsilon\equiv1\ {\rm in\ }[0,\frac1\varepsilon]\ \ \ \ {\rm and}\ \ \ \ \chi_\varepsilon\equiv0\ {\rm in}\ [\frac2\varepsilon,\infty),
\ee
and letting
\be\label{regular-F}
F_\varepsilon(s)=\int^s_0\chi_\varepsilon(\sigma)d\sigma,\ \ \ \ \ \ s\geq0,
\ee
for $\varepsilon\in(0,1)$. Then $F_\varepsilon\in C^\infty([0,\infty))$ satisfies
\be\label{eq1-F}
0\leq F_\varepsilon(s)\leq s\ \ \ \ {\rm and}\ \ \ \ 0\leq F'_\varepsilon(s)\leq1,\ \ \ \ \ \ {\rm for\ all\ }s\geq0
\ee
as well as
\be\label{eq2-F}
F_\varepsilon(s)\nearrow s\ \ \ \ {\rm and}\ \ \ \ F'_\varepsilon(s)\nearrow1,\ \ \ \ \ \ {\rm for\ all\ }s>0,\ \ \ \ \ \ {\rm as}\ \varepsilon\searrow0.
\ee
According to the idea from \cite{Winkler-JDE-2018} (see also \cite{Tao&Winkler-AIHP-2013,Zhang&Li-JDE-2015}), let us consider the approximate variants of (\ref{CS}), (\ref{eq-i}) and (\ref{eq-bc}).
\begin{eqnarray}\label{model-approximate}
  \left\{\begin{array}{lcl}
     \medskip
       \partial_t{n_\varepsilon}+{u_\varepsilon}\cdot\nabla {n_\varepsilon}=\nabla\cdot\left((|\nabla {n_\varepsilon}|^2+\varepsilon)^\frac{p-2}2\nabla {n_\varepsilon}\right)-\nabla\cdot({n_\varepsilon}{F_\varepsilon}'({n_\varepsilon})\nabla {c_\varepsilon}),&& x\in\Omega,\ t>0,\\
     \medskip
     \partial_t {c_\varepsilon}+{u_\varepsilon}\cdot\nabla {c_\varepsilon}=\Delta {c_\varepsilon}-{F_\varepsilon}({n_\varepsilon})c_\varepsilon,&& x\in\Omega,\ t>0,\\
     \medskip
     \partial_t{u_\varepsilon}=\Delta {u_\varepsilon}+\nabla P_\varepsilon+{n_\varepsilon}\nabla\phi,&& x\in\Omega,\ t>0,\\
     \medskip
     \nabla\cdot {u_\varepsilon}=0,&& x\in\Omega,\ t>0,\\
     \medskip
     \frac{\partial {n_\varepsilon}}{\partial \nu}=\frac{\partial {c_\varepsilon}}{\partial \nu}=0,\ {u_\varepsilon}=0,  && x\in\partial\Omega,\ t>0,\\
     \medskip
     {n_\varepsilon}(x,0)=n_0,\ {c_\varepsilon}(x,0)=c_0,\ {u_\varepsilon}(x,0)=u_0,\   && x\in\Omega
  \end{array}\right.
\end{eqnarray}
for $\varepsilon\in(0,1)$, where the families of approximate initial data $n_{0\varepsilon}\geq 0$, $c_{0\varepsilon}\geq 0$ and $u_{0\varepsilon}$ have the properties
\begin{eqnarray}\label{approximate-initialdata-n}
  \left\{\begin{array}{lcl}
     \medskip
     n_{0\varepsilon}\in C^\infty_0(\Omega),\ \ \int_\Omega n_{0\varepsilon}=\int_\Omega n_0,\ \ \ \ {\rm{for\ all}}\ \varepsilon\in(0,1),\\
     \medskip
     \{n_{0\varepsilon}\}_{\eps\in(0,1)}\ {\rm uniformly\ bounded\ in}\ L^\infty(\Omega),\\
     \medskip
     n_{0\varepsilon}\rightarrow n_0\ \ {\rm{in}}\ L^2(\Omega)\ \ \ \ as\ \ \varepsilon\searrow0,
  \end{array}\right.
\end{eqnarray}
\begin{eqnarray}\label{approximate-initialdata-c}
  \left\{\begin{array}{lcl}
     \medskip
     \sqrt{c_{0\varepsilon}}\in C^\infty_0(\Omega),\ \ \|c_{0\varepsilon}\|_{L^\infty(\Omega)}\leq\|c_{0}\|_{L^\infty(\Omega)},\ \ \ \ {\rm{for\ all}}\ \varepsilon\in(0,1),\\
     \medskip
     \sqrt{c_{0\varepsilon}}\rightarrow \sqrt{c_0}\ \ a.e.\ {\rm{in}}\ \Omega\ \ and\ \ W^{1,2}(\Omega)\ \ \ \ as\ \ \varepsilon\searrow0,
  \end{array}\right.
\end{eqnarray}
and
\begin{eqnarray}\label{approximate-initialdata-u}
  \left\{\begin{array}{lcl}
     \medskip
     u_{0\varepsilon}\in C^\infty_{0,\sigma}(\Omega),\ \ {\rm{with}}\ \ \|u_{0\varepsilon}\|_{L^2(\Omega)}=\|u_{0}\|_{L^2(\Omega)},\ \ \ \ {\rm{for\ all}}\ \varepsilon\in(0,1),\\
     \medskip
     u_{0\varepsilon}\rightarrow u_0\ \ \ {\rm{in}}\ L^2(\Omega)\ \ \ \ as\ \ \varepsilon\searrow0.
  \end{array}\right.
\end{eqnarray}
All the above approximate problems admit globally defined classical solution.
\begin{lem}\label{lem2.1}
Let $p\geq2$, then for each $\varepsilon\in(0,1)$, there exist functions
\begin{eqnarray}
  \left\{\begin{array}{lcl}
     \medskip
       {n_\varepsilon}\in C^0(\overline{\Omega}\times[0,\infty))\cap C^{2,1}(\overline{\Omega}\times(0,\infty)),\\
     \medskip
     {c_\varepsilon}\in C^0(\overline{\Omega}\times[0,\infty))\cap C^{2,1}(\overline{\Omega}\times(0,\infty)),\ \ \ \ \\
     \medskip
     {u_\varepsilon}\in C^0(\overline{\Omega}\times[0,\infty);\mathbb{R}^3)\cap C^{2,1}(\overline{\Omega}\times(0,\infty);\mathbb{R}^3),\\
     \medskip
     P_\varepsilon\in C^{1,0}(\overline{\Omega}\times(0,\infty)),
  \end{array}\right.
\end{eqnarray}
such that $({n_\varepsilon},{c_\varepsilon},{u_\varepsilon},\zzh P)$ solves (\ref{model-approximate}) classically in $\Omega\times(0,\infty)$, and such that $\zzh n$ and $\zzh c$ are nonnegative in $\Omega\times(0,\infty)$.
\end{lem}
\begin{proof}
Through the well-established methods involving the Schauder fixed point theorem and the standard regularity theories of parabolic equations and the Stokes system (see \cite{Tao&Winkler-AIHP-2013,Winkler-CPDE-2012,Winkler-AIHP-2016,first-paper} for instance), we can see that there exist $T_{max,\eps}\in(0,\infty]$ and at least one classical solution $(\zzh n,\zzh c, \zzh u, \zzh P)$:
\begin{eqnarray*}
  \left\{\begin{array}{lcl}
     \medskip
       {n_\varepsilon}\in C^0(\overline{\Omega}\times[0,T_{max,\varepsilon}))\cap C^{2,1}(\overline{\Omega}\times(0,T_{max,\varepsilon})),\\
     \medskip
     {c_\varepsilon}\in C^0(\overline{\Omega}\times[0,T_{max,\varepsilon}))\cap C^{2,1}(\overline{\Omega}\times(0,T_{max,\varepsilon})),\\
     \medskip
     {u_\varepsilon}\in C^0(\overline{\Omega}\times[0,T_{max,\varepsilon});\mathbb{R}^3)\cap C^{2,1}(\overline{\Omega}\times(0,T_{max,\varepsilon});\mathbb{R}^3),\\
     \medskip
     P_\varepsilon\in C^{1,0}(\overline{\Omega}\times(0,T_{max,\varepsilon})),
  \end{array}\right.
\end{eqnarray*}
which is such that $\zzh n\geq0$, $\zzh c\geq0$ in $\Omega\times(0,T_{max,\eps})$ and $\zzh c\in C^0([0,T_{max,\eps}),W^{1,q}(\Omega))$ for all $q\geq1$.  Moreover, if $T_{max,\varepsilon}<\infty$, then
\be
\|{n_\varepsilon}(\cdot,t)\|_{L^\infty(\Omega)}+\|{c_\varepsilon}(\cdot,t)\|_{W^{1,q}(\Omega)}+\|A^\alpha{u_\varepsilon}(\cdot,t)\|_{L^2(\Omega)}\rightarrow\infty,\ \ t\nearrow T_{max,\varepsilon}\nonumber
\ee
for all $q>3$ and $\alpha>\frac34$. And then using the similar arguments in \cite[Lemma 4.1]{first-paper} ( for the case of chemotaxis-Navier-Stokes system involving $p$-Laplacian diffusion) we actually have $T_{max,\eps}=\infty$.
\end{proof}

The following two estimates for the classical solutions ${n_\varepsilon}$ and ${c_\varepsilon}$ are basic but important in the proof of our result.
\begin{lem}\label{lem2.2}
For each $\varepsilon\in(0,1)$, the solution of (\ref{model-approximate}) satisfies
\be
\int_\Omega{n_\varepsilon}(\cdot,t)=\int_\Omega n_0,\ \ \ \ for\ all\ t>0\label{eq-n-mass-conservation}
\ee
and
\be
\|{c_\varepsilon}(\cdot,t)\|_{L^\infty(\Omega)}\leq\| c_0\|_{L^\infty(\Omega)}=:s_0,\ \ \ \ for\ all\ t>0.\label{eq-c-maxmum-principle-s0}
\ee
\end{lem}
\begin{proof}
Integrating the first equation in (\ref{model-approximate}) we obtain (\ref{eq-n-mass-conservation}).
And an application of the maximum principle to the second equation in (\ref{model-approximate}) gives (\ref{eq-c-maxmum-principle-s0}).
\end{proof}

\section{An energy-type inequality}
In this section, we derive the following energy-type inequality of the approximate problems (\ref{model-approximate}) and give some consequence.

\begin{lem}\label{lem3.1} Assume $p\geq2$, then there exist $\kappa>0$ and $C>0$ (independent of $\eps$) such that
\bes
&&\frac d{dt}\left\{\int_\Omega{n_\varepsilon}\ln {n_\varepsilon}+\frac12\int_\Omega\frac{|\nabla\zzh c|^2}{\zzh c}+\kappa\jfo|\zzh u|^2\right\}\nonumber\\
&&\ \ \ \ \ \ \ \ \ \ \ \ \ \ \ \ +\frac1C\left\{\int_\Omega{n_\varepsilon}\ln {n_\varepsilon}+\frac12\int_\Omega\frac{|\nabla\zzh c|^2}{\zzh c}+\kappa\jfo|\zzh u|^2\right\}\nonumber\\
&&\ \ \ \ \ \ \ \ \ \ \ \ \ \ \ \ +\frac1C\left\{\jfo|\nabla\zzh n^{1-1/p}|^p+\int_\Omega\frac{|\nabla\zzh c|^4}{\zzh c^3}+\jfo|\nabla\zzh u|^2\right\}\leq C\sjjl.\label{eq-lem3.1}
\ees
\end{lem}
\begin{proof}
By means of straightforward computation using the first two equations in (\ref{model-approximate})(cf. \cite{Winkler-CPDE-2012} for details), we obtain the identity
\begin{align}
&\frac d{dt}\left\{\int_\Omega{n_\varepsilon}\ln {n_\varepsilon}+\frac12\int_\Omega\frac{|\nabla\zzh c|^2}{\zzh c}\right\}+\int_\Omega(|\nabla{n_\varepsilon}|^2+\varepsilon)^\frac{p-2}2\frac{|\nabla{n_\varepsilon}|^2}{{n_\varepsilon}}+\jfo\zzh c|D^2 \ln\zzh c|^2\nonumber\\
&\ \ \ \ \ \ \ \ \ \ \ \ \ \ \ \ \ \ \ \ =-\frac12\jfo\frac{|\nabla\zzh c|^2}{\zzh c^2}(\zzh u\cdot\nabla\zzh c)+\jfo\frac{\Delta\zzh c}{\zzh c}(\zzh u\cdot\nabla\zzh c)\nonumber\\
&\ \ \ \ \ \ \ \ \ \ \ \ \ \ \ \ \ \ \ \ \ \ \ -\frac12\jfo\zzh F(\zzh n)\frac{|\nabla\zzh c|^2}{\zzh c}+\frac12\int_{\partial\Omega}\frac1{\zzh c}\frac{\partial|\nabla\zzh c|^2}{\partial\nu}\sjj.\label{eq1-lem3.1}
\end{align}
Using the inequalities
\bess
\int_\Omega(|\nabla{n_\varepsilon}|^2+\varepsilon)^\frac{p-2}2\frac{|\nabla{n_\varepsilon}|^2}{{n_\varepsilon}}\geq\jfo \frac{|\nabla\zzh n|^p}{\zzh n}=\left(\frac p{p-1}\right)^p\jfo|\nabla\zzh n^{1-1/p}|^p\sjj
\eess
and
$\int_\Omega\frac{|\nabla\zzh c|^4}{\zzh c^3}\leq C_1\jfo\zzh c|D^2\ln \zzh c|^2$ for all $t>0$ (\cite[Lemma 3.3]{Winkler-CPDE-2012}) with $C_1=(2+\sqrt 3)^2$, we can derive from (\ref{eq1-lem3.1}) that
\begin{align}
&\frac d{dt}\left\{\int_\Omega{n_\varepsilon}\ln {n_\varepsilon}+\frac12\int_\Omega\frac{|\nabla\zzh c|^2}{\zzh c}\right\}+\left(\frac p{p-1}\right)^p\jfo|\nabla\zzh n^{1-1/p}|^p+\frac1{C_1}\int_\Omega\frac{|\nabla\zzh c|^4}{\zzh c^3}\nonumber\\
&\ \ \ \ \ \ \ \ \ \ \ \ \ \ \ \ \ \ \ \ \leq-\frac12\jfo\frac{|\nabla\zzh c|^2}{\zzh c^2}(\zzh u\cdot\nabla\zzh c)+\jfo\frac{\Delta\zzh c}{\zzh c}(\zzh u\cdot\nabla\zzh c)\nonumber\\
&\ \ \ \ \ \ \ \ \ \ \ \ \ \ \ \ \ \ \ \ \ \ \ -\frac12\jfo\zzh F(\zzh n)\frac{|\nabla\zzh c|^2}{\zzh c}+\frac12\int_{\partial\Omega}\frac1{\zzh c}\frac{\partial|\nabla\zzh c|^2}{\partial\nu}\sjj.\label{sdafahajjj}
\end{align}

From the facts that the nonnegativity of $\zzh F$ and $\frac{\partial|\nabla\zzh c|^2}{\zzh c}\leq 0$ on $\partial\Omega\times(0,\infty)$ since $\Omega$ due to convex (\cite[Lemma 2.I.1]{Lions-1980-ARMA}), we know that the last two summands on the right of (\ref{eq1-lem3.1}) are nonpositive. This shows that (\ref{sdafahajjj}) leads to theinequality
\begin{align}
&\frac d{dt}\left\{\int_\Omega{n_\varepsilon}\ln {n_\varepsilon}+\frac12\int_\Omega\frac{|\nabla\zzh c|^2}{\zzh c}\right\}+\left(\frac p{p-1}\right)^p\jfo|\nabla\zzh n^{1-1/p}|^p+\frac1{C_1}\int_\Omega\frac{|\nabla\zzh c|^4}{\zzh c^3}\nonumber\\
&\ \ \ \ \ \ \ \ \ \ \ \ \ \ \ \ \ \ \ \ \leq-\frac12\jfo\frac{|\nabla\zzh c|^2}{\zzh c^2}(\zzh u\cdot\nabla\zzh c)+\jfo\frac{\Delta\zzh c}{\zzh c}(\zzh u\cdot\nabla\zzh c)\sjj.\label{kfjsaggfarga}
\end{align}

To estimate the right hand side of (\ref{kfjsaggfarga}), we make the following computation by using integration by parts, the identity $\jfo\frac1{\zzh c}\nabla\zzh c\cdot(D^2\zzh c\cdot\zzh u)=\frac12\jfo\frac{|\nabla\zzh c|^2}{\zzh c^2}(\zzh u\cdot\nabla\zzh c)$ and Young's inequality.

\bes\label{eq3-lem3.1}
&&\ \ \ -\frac12\jfo\frac{|\nabla\zzh c|^2}{\zzh c^2}(\zzh u\cdot\nabla\zzh c)+\jfo\frac{\Delta\zzh c}{\zzh c}(\zzh u\cdot\nabla\zzh c)\nonumber\\
&&=\frac12\jfo\frac{|\nabla\zzh c|^2}{\zzh c^2}(\zzh u\cdot\nabla\zzh c)-\jfo\frac1{\zzh c}\nabla\zzh c\cdot(D^2\zzh c\cdot\zzh u)-\jfo\frac1{\zzh c}\nabla\zzh c\cdot(\nabla\zzh u\cdot\nabla\zzh c)\nonumber\\
&&=-\jfo\frac1{\zzh c}\nabla\zzh c\cdot(\nabla\zzh u\cdot\nabla\zzh c)\nonumber\\
&&\leq\frac1{2C_1}\jfo\frac{|\nabla\zzh c|^4}{\zzh c^3}+\frac{C_1}2\jfo|\zzh c|\cdot|\nabla\zzh u|^2\nonumber\\
&&\leq\frac1{2C_1}\jfo\frac{|\nabla\zzh c|^4}{\zzh c^3}+C_2\jfo|\nabla\zzh u|^2\sjj
\ees
with $C_2=\frac{C_1}2s_0$, where we have used (\ref{eq-c-maxmum-principle-s0}) in the last inequality. So we obtain
\begin{align}
&\frac d{dt}\left\{\int_\Omega{n_\varepsilon}\ln {n_\varepsilon}+\frac12\int_\Omega\frac{|\nabla\zzh c|^2}{\zzh c}\right\}+\left(\frac p{p-1}\right)^p\jfo|\nabla\zzh n^{1-1/p}|^p+\frac1{2C_1}\int_\Omega\frac{|\nabla\zzh c|^4}{\zzh c^3}\nonumber\\
&\ \ \ \ \ \ \ \ \ \ \ \ \ \ \ \ \ \ \ \ \leq C_2\jfo|\nabla\zzh u|^2\sjj.\label{kdregasdgvv}
\end{align}

To treat the right hand side of (\ref{kdregasdgvv}), we test the third equation in (\ref{model-approximate}) by $\zzh u$ and use the Sobolev embedding $W^{1,2}_0(\Omega)\hookrightarrow \lp{6}$ to obtain
\bes
\frac12\frac d{dt}\jfo|\zzh u|^2+\jfo|\nabla\zzh u|^2&=&\jfo\zzh n\zzh u\cdot\nabla\phi\nonumber\\
&\leq&\|\nabla\phi\|_{\lp\infty}\|\zzh u\|_{\lp 6}\|\zzh n\|_{\lp{\frac65}}\nonumber\\
&\leq &C_3\|\nabla\zzh u\|_{\lp 2}\|\zzh n^{1-1/p}\|^{\frac p{p-1}}_{\lp{\frac{6p}{5(p-1)}}}\nonumber\\
&\leq&\frac12\jfo|\nabla\zzh u|^2+\frac{C_3^2}2\|\zzh n^{1-1/p}\|^{\frac {2p}{p-1}}_{\lp{\frac{6p}{5(p-1)}}}\sjj,\label{eq4-lem3.1}
\ees
where Cauchy-Schwarz inequality has been used in the last inequality. Since $p\geq2>\frac{11}7$, we have $\frac1p-\frac13\leq\frac{5(p-1)}{6p}$ and hence $W^{1,p}(\Omega)\hookrightarrow L^{\frac{6p}{5(p-1)}}(\Omega)\hookrightarrow L^{\frac{p}{(p-1)}}(\Omega)$. Now the Gagliardo-Nirenberg inequality together with (\ref{eq-n-mass-conservation}) shows that there exist positive constants $C_4>0$ and $C_5>0$ such that
\bes\label{eq5-lem3.1}
\frac{C_3^2}2\|\zzh n^{1-1/p}\|^{\frac {2p}{p-1}}_{\lp{\frac{6p}{5(p-1)}}}&\leq& C_4\left\{\|\nabla\zzh n^{1-1/p}\|^{\theta\frac{2p}{p-1}}_{\lp p}\|\zzh n^{1-1/p}\|_{\lp{\frac p{p-1}}}^{(1-\theta)\frac{2p}{p-1}}+\|\zzh n^{1-1/p}\|_{\lp{\frac p{p-1}}}^{\frac{2p}{p-1}}\right\}\nonumber\\
&=& C_4\left\{\|\nabla\zzh n^{1-1/p}\|^{\theta\frac{2p}{p-1}}_{\lp p}\|\zzh n\|_{\lp{1}}^{2(1-\theta)}+\|\zzh n\|_{\lp 1}^{2}\right\}\nonumber\\
&\leq&C_5\|\nabla\zzh n^{1-1/p}\|^{\theta\frac{2p}{p-1}}_{\lp p}+C_5\nonumber\\
&=& C_5\|\nabla\zzh n^{1-1/p}\|^{\frac{p}{2(2p-3)}}_{\lp p}+C_5\sjj,
\ees
where $\theta=\frac{p-1}{4(2p-3)}=\frac18+\frac1{8(2p-3)}\in(\frac18,\frac14]$ thanks to $p\geq2$.
Noticing that $p\geq2>\frac74$ implies $\frac{p}{2(2p-3)}<p$, we then use Young's inequality along with (\ref{eq4-lem3.1}) and (\ref{eq5-lem3.1}) to see that there exist positive constant $C_6>0$ such that for all $t>0$
\bes
\frac d{dt}\jfo|\zzh u|^2+\jfo|\nabla\zzh u|^2&\leq&2C_5\|\nabla\zzh n^{1-1/p}\|^{\frac{p}{2(2p-3)}}_{\lp p}+2C_5\nonumber\\
&\leq&\frac1{2(C_2+1)}\left(\frac p{p-1}\right)^p\|\nabla\zzh n^{1-1/p}\|^p_{\lp p}+C_6\label{eq7-lem3.1}
\ees
Combining (\ref{eq7-lem3.1}) with (\ref{kdregasdgvv}), we obtain
\bes
&&\frac d{dt}\left\{\int_\Omega{n_\varepsilon}\ln {n_\varepsilon}+\frac12\int_\Omega\frac{|\nabla\zzh c|^2}{\zzh c}+(C_2+1)\jfo|\zzh u|^2\right\}\nonumber\\
&&+\frac12\left(\frac p{p-1}\right)^p\jfo|\nabla\zzh n^{1-1/p}|^p+\frac1{2C_1}\jfo\frac{|\nabla\zzh c|^4}{\zzh c^3}+\jfo|\nabla\zzh u|^2\nonumber\\
&&\leq2(C_2+1)C_6\sjj.\label{eq8-lem3.1}
\ees

Using the elementary inequality $z\ln z\leq2 z^{\frac65}$ for all $z\geq0$, (\ref{eq5-lem3.1}) and Young's inequality, we have
\bes
\jfo \zzh n\ln\zzh n&\leq& 2\jfo \zzh n^\frac65\nonumber\\
&=&2\|\zzh n^{1-1/p}\|^{\frac {6p}{5(p-1)}}_{\lp{\frac{6p}{5(p-1)}}}\nonumber\\
&\leq&2\|\zzh n^{1-1/p}\|^{\frac {2p}{p-1}}_{\lp{\frac{6p}{5(p-1)}}}+2\nonumber\\
&\leq&\frac{4}{C_3^2}\left(C_5\|\nabla\zzh n^{1-1/p}\|^{\frac{p}{2(2p-3)}}_{\lp p}+C_5\right)+2\nonumber\\
&\leq&\frac{4}{C_3^2}\left(C_5\|\nabla\zzh n^{1-1/p}\|^p_{\lp p}+2C_5\right)+2\nonumber\\
&=&C_7\|\nabla\zzh n^{1-1/p}\|^p_{\lp p}+C_8\sjj,\label{ofdisahnfeasdf}
\ees
where $C_7=\frac{4C_5}{C_3^2}$ and $C_8=\frac{8C_5}{C_3^2}+2$. And using Young's inequality and (\ref{eq-c-maxmum-principle-s0}), we have
\bes
\frac12\int_\Omega\frac{|\nabla\zzh c|^2}{\zzh c}\leq\frac14\jfo\frac{|\nabla\zzh c|^4}{\zzh c^3}+\frac14\jfo\zzh c\leq\frac14\jfo\frac{|\nabla\zzh c|^4}{\zzh c^3}+\frac{s_0}4\Omega.\label{iquyfohads}
\ees
From Poincar$\rm{\acute{e}}$ inequality, we obtain some constant $C_9>0$ such that
\be
\jfo|\zzh u|^2\leq C_9\jfo|\nabla\zzh u|^2\sjj.\label{sfilhuaflad}
\ee
Combining (\ref{ofdisahnfeasdf})-(\ref{sfilhuaflad}), we immediately arrive at
\bess
&&\int_\Omega{n_\varepsilon}\ln {n_\varepsilon}+\frac12\int_\Omega\frac{|\nabla\zzh c|^2}{\zzh c}+\kappa\jfo|\zzh u|^2\\
&\leq& C_7\|\nabla\zzh n^{1-1/p}\|^p_{\lp p}+\frac14\jfo\frac{|\nabla\zzh c|^4}{\zzh c^3}+\kappa C_9\jfo|\nabla\zzh u|^2+C_8+\frac{s_0}4\Omega\sjj.
\eess
This together with (\ref{eq8-lem3.1}) readily establishes (\ref{eq-lem3.1}) upon evident choices of $\kappa$ and $C$.
\end{proof}

The following is the direct consequence of Lemma \ref{lem3.1} and will be used frequently in the sequel.
\begin{lem}\label{lem3.2}
There exist $C>0$ such that for all $\varepsilon\in(0,1)$,
\be
\int^{t+1}_t\jfo|\nabla\zzh n^{1-1/p}|^p\leq C\sjl
\ee
and
\be\label{cforq2}
\int^{t+1}_t\jfo|\nabla\zzh c|^4\leq C\sjl
\ee
as well as
\be
\int^{t+1}_t\jfo|\nabla\zzh u|^2\leq C\sjl.
\ee
\end{lem}
\begin{proof}
All inequalities immediately result from an integration of (\ref{eq-lem3.1}) because of (\ref{eq-c-maxmum-principle-s0}) and the fact that $\jfo\zzh n\ln \zzh n\geq-\frac{|\Omega|}e$ for all $t\geq0$.
\end{proof}
\section{Preparing an inductive argument}
With Lemma \ref{lem2.2} and Lemma \ref{lem3.1} at hand, we can improve the integrability of $\zzh n$ step by step.
\begin{lem}\label{lem4.1}
Let $m_*\geq1$, $q\geq2$, $p>\frac{25}{12}$, $m>m_*$ be such that
\be
\frac{2q(p-1)-p}{2q(p-1)}m_*+\frac{p-2}{p-1}<m\leq \frac{2q(p-1)-p}3 m_*+(2q-1)(p-2).\label{m-constant-condition-lem4.1}
\ee
Then for all $K>0$ there exist $C=C(m_*,m,p,q,K)>0$ such that if for some $\varepsilon\in(0,1)$ there hold
\be
\int_\Omega n_\varepsilon^{m_*}(\cdot,t)\leq K,\ \ \ \ {for\ all}\ t\geq0\label{n-m-0-condition-lem4.1}
\ee
and
\be
\int^{t+1}_t\int_\Omega|\nabla{c_\varepsilon}(\cdot,s)|^{2q}ds\leq K,\ \ \ \ {for\ all}\ t\geq0,\label{nabla-c-4-condition-lem4.1}
\ee
then we have
\be
\int_\Omega n_\varepsilon^m(\cdot,t)\leq C,\ \ \ \ {for\ all}\ t\geq0\label{n-m-result-lem4.1}
\ee
and
\be
\int^{t+1}_t\int_\Omega|\nabla n_\varepsilon^{\frac{m-2}p+1}|^p\leq C,\ \ \ \ {for\ all}\ t\geq0.\label{nabla-n-m-2-result-lem4.1}
\ee
\end{lem}

\begin{proof} Testing the first equation in (\ref{model-approximate}) by $ n_\varepsilon^{m-1}$ and using Young's inequality along with (\ref{eq-c-maxmum-principle-s0}) we can see that there are positive constants $C_1>0$ and $C_2>0$ such that for all $t > 0$,
\begin{eqnarray*}
\frac1m\frac{d}{dt}\int_\Omega n_\varepsilon^m+(m-1)\int_\Omega|\nabla n_\varepsilon|^p n_\varepsilon^{m-2}&\leq&C_1(m-1)\int_\Omega n_\varepsilon^{m-1}|\nabla{c_\varepsilon}|\cdot |\nabla n_\varepsilon|\\
&\leq&(m-1)\int_\Omega\left(\frac12|\nabla n_\varepsilon|^p n_\varepsilon^{m-2}+C_2( n_\varepsilon^{\frac{2-m}p+m-1}|\nabla{c_\varepsilon}|)^{p'}\right)\\
\end{eqnarray*}
where $p'=\frac{p}{p-1}$, so that
\begin{eqnarray}
\frac 1m\frac{d}{dt}\int_\Omega n_\varepsilon^m+
\frac{m-1}2\int_\Omega|\nabla n_\varepsilon|^p n_\varepsilon^{m-2}\leq C_2(m-1)\int_\Omega( n_\varepsilon^{\frac{2-m}p+m-1}|\nabla{c_\varepsilon}|)^{p'},\ \ \ \ {\rm{for\ all}}\ t>0.\label{eq1-lem6.1-proof}
\end{eqnarray}
Let ${\beta}=\left({\frac{2-m}p+m-1}\right)*p'=m+\frac{1}{p-1}-1$
and $\widetilde{m}=\frac{m-2}p+1$. Applying Young's inequality to (\ref{eq1-lem6.1-proof}), we obtain $C_2>0$ such that
\begin{eqnarray}
\frac 1m\frac{d}{dt}\int_\Omega n_\varepsilon^m+\frac{(m-1)}{2\widetilde{m}^p}\int_\Omega|\nabla n_\varepsilon^{\widetilde{m}}|^p&\leq& C_2(m-1)\int_\Omega n_\varepsilon^{{\beta}}|\nabla{c_\varepsilon}|^{p'}\nonumber\\
&\leq& C_2(m-1)\int_\Omega(\zzh n^{\beta\frac{2q}{2q-p'}}+|\nabla\zzh c|^{2q}),\ \ \ \ {\rm{for\ all}}\ t>0.\label{eq4.8}
\end{eqnarray}

Denote $A_1=\frac{2q\beta}{2q-p'}$. If $A_1\leq m_*$, then we can apply H$\rm{\ddot{o}lder}$ inequality together with (\ref{n-m-0-condition-lem4.1}) to (\ref{eq4.8}) to obtain $C_3>0$ fulfilling
\begin{eqnarray}
\frac 1m\frac{d}{dt}\int_\Omega n_\varepsilon^m+\frac{(m-1)}{2\widetilde{m}^p}\int_\Omega|\nabla n_\varepsilon^{\widetilde{m}}|^p
&\leq& C_3(m-1)(1+\int_\Omega|\nabla\zzh c|^{2q}),\ \ \ \ {\rm{for\ all}}\ t>0.\label{eq-key-one}
\end{eqnarray}

If $A_1>m_*$. Due to our assumption $p\geq2$, $q\geq2$ and $m>\frac{p-2}{p-1}$ we have
\be
\frac{\widetilde{m}}{A_1}-\left(\frac1p-\frac13\right)=\frac{m (2 q(p-1)-3)+(p-2) (4 q-3)}{6 q (m (p-1)-p+2)}>0.\nonumber
\ee
Thus $W^{1,p}(\Omega)\hookrightarrow L^{\frac{A_1}{\widetilde{m}}}(\Omega)\hookrightarrow L^{\frac{m_*}{\widetilde{m}}}(\Omega)$, accordingly the Gagliardo-Nirenberg inequality together with (\ref{eq-n-mass-conservation}) provides $C_4>0$ and $C_5>0$ such that
\begin{eqnarray}
\int_\Omega\zzh n^{\beta\frac{2q}{2q-p'}}&=&\left\|\zzh n^{\widetilde{m}}\right\|^{\frac{A_1}{\widetilde{m}}}_{\frac{A_1}{\widetilde{m}}}\nonumber\\
&\leq& C_4\|\nabla\zzh n^{\widetilde{m}}\|_p^{{\frac{A_1}{\widetilde{m}}}\theta_1}\left\|\zzh n^{\widetilde{m}}\right\|^{{\frac{A_1}{\widetilde{m}}}(1-\theta_1)}_{\frac{m_*}{\widetilde{m}}}+C_4\left\|\zzh n^{\widetilde{m}}\right\|^{{\frac{A_1}{\widetilde{m}}}}_{\frac{m_*}{\widetilde{m}}}\nonumber\\
&\leq&C_5(\|\nabla\zzh n^{\widetilde{m}}\|_p^{{\frac{A_1}{\widetilde{m}}}\theta_1}+1),\ \ \ \ {\rm{for\ all}}\ t>0,\label{eq4.10}
\end{eqnarray}
where
\be
\theta_1=\frac{3 (m+p-2) (2 q (m (p-1)-p ({m_*}+1)+{m_*}+2)+p {m_*})}{2 q (m (p-1)-p+2) (3 (m+p-2)+(p-3) {m_*})}\nonumber
\ee
 satisfies $\frac{\widetilde{m}}{A_1}=\theta_1(\frac1p-\frac13)+(1-\theta_1)\frac{\widetilde{m}}{m_*}$. A routine computation shows that $\theta_1\in(0,1)$. Indeed, the left inequality of (\ref{m-constant-condition-lem4.1}) ensures that
 \be
 2 q (m (p-1)-p ({m_*}+1)+{m_*}+2)+p {m_*}>0.\label{eqsssskkl}
 \ee
Since $q\geq2$, the left inequality of (\ref{m-constant-condition-lem4.1}) implies
\be
m>\frac{(2q(p-1)-p){m_*}}{2 q(p-1)}+\frac{p-2}{p-1}=\frac{(p-1-\frac {p}{2q}){m_*}}{(p-1) }+\frac{p-2}{p-1}\geq\frac{(3p-4){m_*}}{4 (p-1)}+\frac{p-2}{p-1},\nonumber
\ee
hence we obtain
\begin{eqnarray}
3 (m+p-2)+(p-3) {m_*}&=&3\left(m- (m_* (1-\frac{p}{3})-p+2)\right)\nonumber\\
&>&3(\frac{m_* (3 p-4)}{4 p-4}+\frac{p-2}{p-1}-(m_* (1-\frac{p}{3})-p+2))\nonumber\\
&=&3\frac{ (m_* (4 p-7)+12 (p-2))p}{12 (p-1)}>0.\label{eqsssskkll}
\end{eqnarray}
And our assumption $m>\frac{p-2}{p-1}$ and $p\geq1$ ensure that $m (p-1)-p+2>0,$ this together with (\ref{eqsssskkl}) and (\ref{eqsssskkll}) shows that $\theta_1>0$. On the other hand, we claim that
\be
\theta_1-1=\frac{p {m_*} (2 q (-(m+2) p+m+4)+3 (m+p-2))}{2 q (m (p-1)-p+2) (3 (m+p-2)+(p-3) {m_*})}<0.\label{kjhfdlsaknda}
\ee
In fact, from $q\geq2$ and $p\geq2>\frac{17}9$ we have
\bess
2 q (-(m+2) p+m+4)+3 (m+p-2)&=&- (2 q(p-1) -3)m-(p-2) (4 q-3)\\
&\leq& - (2 q(p-1) -3)-(p-2) (4 q-3)\\
&=&-2q(3p-5)+3(p-1)\\
&\leq&-4(3p-5)+3(p-1)\\
&=&17-9p<0,
\eess
this together with $m>\frac{p-2}{p-1}$, (\ref{eqsssskkll}) and (\ref{kjhfdlsaknda}) confirm that $\theta_1<1$.

Now, the right inequality of (\ref{m-constant-condition-lem4.1}) ensures that
\begin{eqnarray*}
\frac{A_1}{\widetilde{m}}\theta_1- p&=&\frac{p^2 (3 m-p (2 q-1) ({m_*}+3)+2 (q ({m_*}+6)-3))}{(2 (p-1) q-p) (3 (m+p-2)+(p-3) {m_*})},\\
&=&\frac{3p^2 \left(m-\left(\frac{2 q(p-1) -p}{3} {m_*} +(p-2) (2 q-1)\right)\right)}{(2 (p-1) q-p) (3 (m+p-2)+(p-3) {m_*})},\\
&<&0,
\end{eqnarray*}
this together with (\ref{eq4.10}) and Young's inequality shows that
\bes\label{ldfajflaef}
C_2(m-1)\int_\Omega\zzh n^{\beta\frac{2q}{2q-p'}}&\leq& C_2C_5(m-1)(\|\nabla\zzh n^{\widetilde{m}}\|_p^{{\frac{A_1}{\widetilde{m}}}\theta_1}+1)\nonumber\\
&\leq&\frac{(m-1)}{4\widetilde{m}^p}\|\nabla\zzh n^{\widetilde{m}}\|_p^p+C_6(m-1)\sjj,
\ees
with some $C_6\geq C_2>0$.

Substituting (\ref{ldfajflaef}) into (\ref{eq4.8}), we obtain
\begin{eqnarray}
\frac 1m\frac{d}{dt}\int_\Omega n_\varepsilon^m+\frac{(m-1)}{4\widetilde{m}^p}\int_\Omega|\nabla n_\varepsilon^{\widetilde{m}}|^p
&\leq& C_6(m-1)(1+\int_\Omega|\nabla\zzh c|^{2q}),\ \ \ \ {\rm{for\ all}}\ t>0.\label{eq-key-two}
\end{eqnarray}
We can generate a linear absorption term in (\ref{eq-key-two}). Noticing $\widetilde{m}=\frac{m-2}p+1$, we have the inequality
\be
\frac{\widetilde{m}}{m}-\left(\frac1p-\frac13\right)=\frac{p-2}{m p}+\frac{1}{3}>0,
\ee
so there holds $W^{1,p}(\Omega)\hookrightarrow L^{\frac{m}{\widetilde{m}}}(\Omega)\hookrightarrow L^{\frac{m_*}{\widetilde{m}}}(\Omega)$. In addition, define $\theta_2= \frac{3 (m+p-2) (m-{m_*})}{m(3(m-m_*)+3(p-2)+pm_*)}$ satisfying $\frac{\widetilde{m}}{m}=\theta_2(\frac1p-\frac13)+(1-\theta_2)\frac{\widetilde{m}}{m_*}$ and $\theta_2\in(0,1)$ thanks to $\theta_2-1=-\frac{{m_*} (m+3( p-2))}{m(3(m-m_*)+3(p-2)+pm_*)}<0$. Using the Gagliardo-Nirenberg inequality, (\ref{n-m-0-condition-lem4.1}) and Young's inequality, we obtain $C_7>0$ and $C_8>0$ such that
\bes
\int_\Omega n_\varepsilon^m&=&\|n_\varepsilon^{\widetilde{m}}\|_\frac{m}{\widetilde{m}}^\frac{m}{\widetilde{m}}\nonumber\\
&\leq& C_7\|\nabla\zzh n^{\widetilde{m}}\|_p^{{\frac{m}{\widetilde{m}}}\theta_2}\left\|\zzh n^{\widetilde{m}}\right\|^{{\frac{m}{\widetilde{m}}}(1-\theta_2)}_{\frac{m_*}{\widetilde{m}}}+C_7\left\|\zzh n^{\widetilde{m}}\right\|^{{\frac{m}{\widetilde{m}}}}_{\frac{m_*}{\widetilde{m}}}\nonumber\\
&\leq&C_8(\|\nabla\zzh n^{\widetilde{m}}\|_p^{{\frac{m}{\widetilde{m}}}\theta_2}+1)\nonumber\\
&\leq&C_8(\|\nabla\zzh n^{\widetilde{m}}\|_p^p+2),\ \ \ \ {\rm{for\ all}}\ t>0,\label{eadfasdfadfsabag}
\ees
where we have used ${\frac{m}{\widetilde{m}}}\theta_2-p=-\frac{p (p {m_*}+3(p-2))}{3 (m+p-2-m_*)+p{m_*}}<0$ in the last inequality.

A combination of (\ref{eq-key-two}) and (\ref{eadfasdfadfsabag}) shows that
\be
\frac 1m \frac d{dt}\jfo\zzh n^m+\frac{(m-1)}{8\widetilde{m}^p C_8}\jfo\zzh n^m+\frac{(m-1)}{8\widetilde{m}^p}\int_\Omega|\nabla n_\varepsilon^{\widetilde{m}}|^p\leq C_6(m-1)(1+\int_\Omega|\nabla\zzh c(\cdot,t)|^{2q})+\frac{(m-1)}{4\widetilde{m}^p },\\
\ \label{alsisafghhhehf}
\ee
for all $t>0$. Let $y(t)=\int_\Omega n_\varepsilon^m(\cdot,t)$, $t>0$, and
$h(t)= C_6(m-1)(1+\int_\Omega|\nabla\zzh c(\cdot,t)|^{2q})+\frac{(m-1)}{4\widetilde{m}^p }$, $t>0$. Then (\ref{alsisafghhhehf}) can be rewritten as
\be
\frac 1my'(t)+\frac{(m-1)}{8\widetilde{m}^p C_8}y(t)+\frac{(m-1)}{8\widetilde{m}^p}\int_\Omega|\nabla n_\varepsilon^{\widetilde{m}}|^p\leq h(t),\ \ \ \ {\rm{for\ all}}\ t>0.\label{alsiefujoaehf}
\ee
In view of (\ref{nabla-c-4-condition-lem4.1}), we have
\be
\int^{t+1}_th(s)ds\leq C_9=C_6(m-1)(1+K)+\frac{(m-1)}{4\widetilde{m}^p }\sjj.\label{askjfhkdaslfh}
\ee
Obviously, (\ref{alsisafghhhehf}) implies
\be
\frac 1m \frac d{dt}\jfo\zzh n^m+\frac{(m-1)}{8\widetilde{m}^p C_8}\jfo\zzh n^m\leq C_6(m-1)(1+\int_\Omega|\nabla\zzh c(\cdot,t)|^{2q})+\frac{(m-1)}{4\widetilde{m}^p }\sjj.\\
\ \label{alsjkrfgsdf}
\ee
According to an elementary lemma on decay in linear first-order ODEs with suitably decaying inhomogeneities (see e.g.\cite[Lemma 3.4]{Winkler-SIAM-2014}), we can deduce  from (\ref{askjfhkdaslfh}) and (\ref{alsjkrfgsdf}) that $y(t)<C_{10}$ for some $C_{10}>0$. The boundedness of $y(t)$ together with (\ref{alsiefujoaehf}) and (\ref{askjfhkdaslfh}) shows that
\be
\frac{(m-1)}{8\widetilde{m}^p}\int^{t+1}_t\int_\Omega|\nabla n_\varepsilon^{\widetilde{m}}|^p\leq \frac 1m C_{10}+\int^{t+1}_th(t)\leq C_{11}=\frac{C_{10}}m+C_9,\ \ \ \ {\rm{for\ all}}\ t>0,\nonumber
\ee
so that indeed both (\ref{n-m-result-lem4.1}) and (\ref{nabla-n-m-2-result-lem4.1}) hold with some suitably large $C=C(m_*,m,p,q,K)>0$.
\end{proof}

\begin{rem} The set of $m$ that satisfies (\ref{m-constant-condition-lem4.1}) in Lemma \ref{lem4.1} is not empty. On the one hand, the right-hand side of (\ref{m-constant-condition-lem4.1}) is larger than the left-hand side:
\begin{eqnarray*}
&&\left(\frac{2q(p-1)-p}3 m_*+(2q-1)(p-2)\right)-\left(\frac{2q(p-1)-p}{2q(p-1)}m_*+\frac{p-2}{p-1}\right)\\
&=& \frac{(2 (p-1) q-p) ((2q(p-1)-3)m_*+6q(p-2))}{6 (p-1) q}>0.
\end{eqnarray*}
One the other hand, the right-hand side of (\ref{m-constant-condition-lem4.1}) is bigger than 1:
\begin{eqnarray*}
&&\left(\frac{2q(p-1)-p}3 m_*+(2q-1)(p-2)\right)-1\\
&=&\frac{1}{3} \left(6q(p-2)+3(1-p)+m_*(2q(p-1)-p)\right)\\
&\geq&\frac{1}{3} \left(6q(p-2)+3(1-p)+2q(p-1)-p\right)\\
&=&\frac{1}{3} \left(2q(4p-7)+3-4p\right)\\
&\geq&\frac{1}{3} \left(4(4p-7)+3-4p\right)\\
&=&\frac{1}{3} \left(12p-25\right)>0,
\end{eqnarray*}
since $p>\frac{25}{12}$.
\end{rem}
\vskip 3mm
Based on Lemma \ref{lem4.1}, we can iteratively get the following conclusion.
\begin{lem} \label{lemfhdaskh}
Let $p>\frac{25}{12}$ and
\begin{eqnarray}
  \left\{\begin{array}{lll}
     r=\frac{9(p-2)}{7-3p},&{} \frac{25}{12}<p<\frac73,\\[1mm]
     r\in(1,\infty),&{} p\geq\frac73.
  \end{array}\right.
\end{eqnarray}
Then for all $m\in[1,r)$ there exist $C=C(m)>0$ such that for all $\varepsilon\in(0,1)$,
\be
\int_\Omega n_\varepsilon^m(\cdot,t)\leq C,\ \ \ \ {for\ all}\ t\geq0\label{n-m-result-lem6.2}
\ee
and
\be
\int^{t+1}_t\int_\Omega|\nabla n_\varepsilon^{\frac{m-2}p+1}|^p\leq C,\ \ \ \ {for\ all}\ t\geq0.\label{nabla-n-m-2-result-lem6.2}
\ee
\end{lem}
\begin{proof} Based on (\ref{eq-n-mass-conservation}) and Lemma \ref{lem3.2}, we can apply Lemma \ref{lem4.1} to $q=2$ and define $(m_k)_{k\in \mathbb{N}_0}\subset \mathbb{R}$ by letting $m_0=1$ and
\be\label{eq-mk}
m_{k+1}=(p-\frac43)m_k+3(p-2),\ \ \ \ {\rm for}\ k\geq0.
\ee
If $p\geq\frac73$, then $m_{k+1}-m_k\geq1$, and hence the sequence $(m_k)_{k\in \mathbb{N}_0}$ is strictly increasing with $m_k\nearrow\infty$ as $k\rightarrow\infty$.

If $p\in(\frac{25}{12},\frac73)$, we can deduce from (\ref{eq-mk}) that
\be
m_k=(p-\frac43)^k\frac{12p-25}{3p-7}+\frac{9(p-2)}{7-3p},\nonumber
\ee
therefore the sequence $(m_k)_{k\in \mathbb{N}_0}$ is strictly increasing with an upper bound $\frac{9(p-2)}{7-3p}$ as $k\rightarrow\infty$ since $p-\frac43\in(0,1)$ and $\frac{12p-25}{3p-7}<0$.
So by means of an interpolation arguement it is clear that we only need to prove (\ref{n-m-result-lem6.2}) and (\ref{nabla-n-m-2-result-lem6.2}) for $m=m_k$ and each $k\in \mathbb{N}_0$. For this purpose, we first point out that the case $k=0$ can be proved by the combination of (\ref{eq-n-mass-conservation}) and Lemma \ref{lem3.2}, so that using mathematical induction, it suffices to prove that if for any $k\in \mathbb{N}_0$ there hold
\be\label{diedai1}
\int_\Omega n_\varepsilon^{m_k}(\cdot,t)\leq C_1(m)\ \ {\rm and}\ \ \int^{t+1}_t\int_\Omega|\nabla n_\varepsilon^{\frac{{m_k}-2}p+1}|^p\leq C_1(m),\ \ \ \ {\rm for\ all}\ t\geq0\ {\rm and}\ {\rm each}\ \varepsilon\in(0,1)
\ee
with some $C_1(m)>0$, then we can find $C_2(m)>0$ satisfying
\be\label{diedai2}
\int_\Omega n_\varepsilon^{m_{k+1}}(\cdot,t)\leq C_2(m)\ \ {\rm and}\ \ \int^{t+1}_t\int_\Omega|\nabla n_\varepsilon^{\frac{{m_{k+1}}-2}p+1}|^p\leq C_2(m),\ \ \ \ {\rm for\ all}\ t\geq0\ {\rm and}\ {\rm each}\ \varepsilon\in(0,1).
\ee
To this end, we note that the requirements (\ref{n-m-0-condition-lem4.1}) and (\ref{nabla-c-4-condition-lem4.1}) from Lemma \ref{lem4.1} are satisfied with $m_*=m_k$ and $q=2$ thanks to the first inequality in (\ref{diedai1}) and (\ref{cforq2}). Applying Lemma \ref{lem4.1} to $m=m_{k+1}$, we arrive at (\ref{diedai2}).
\end{proof}

\section{Improving estimate for $\nabla\zzh c$}\label{sdhalif}
It follows from Lemma \ref{lemfhdaskh} that if $p\geq\frac73$, we can using the Moser iteration method to get the $L^\infty$ estimate of $\zzh n$. But for $p<\frac73$, we couldn't use the Moser iteration, since (\ref{n-m-result-lem6.2}) holds only for $m<\frac{9(p-2)}{7-3p}$. For $p<\frac73$, in order to obtain (\ref{n-m-result-lem6.2}) for all $m\geq1$, we first need to improve the integrability of $\nabla \zzh c$, which enables us to use Lemma \ref{lem4.1}. Noting that in the proof of Lemma \ref{lemfhdaskh}, we use Lemma \ref{lem4.1} only for $q=2$ in (\ref{nabla-c-4-condition-lem4.1}). We first give the improving estimate for $\zzh u$.
\begin{lem}\label{lem6.u}
Let $p\in(\frac{545}{261},\frac73)$. Then there exists $\delta_1(p)>0$ such that for all $m>1$ fulfilling $m>\frac{9 (p-2)}{7-3 p}-\delta_1(p)$ and any $K>0$, if
\be
\jfo \zzh n^m(\cdot,t)\leq K\sjl,\label{ldkfmaldfnlajdfnad}
\ee
then there exist $C(m,p,K)>0$ such that
\be
\jfo |\zzh u(\cdot,t)|^{2(\zs)}\leq C(m,p,K)\sjl.\label{djfkladnvfuaihe}
\ee
\end{lem}
\begin{proof} Let $q=2\left(\frac{m p}{3}+m+p-2\right)$ and
\be
\rho(\tau)=4 (3 + p) \tau^2 - (33 - 6 p) \tau - 18 (p-2),\ \ \ \ \tau\in\mathbb{R},\label{fun-rho}
\ee
then we have
\be
\frac32(\frac1m-\frac1q)-1=-\frac{\rho(m)}{4 m (m (p+3)+3 (p-2))}.\label{lasdkfjlashdfl}
\ee
(\ref{fun-rho}) implies $\rho'(\tau)<0$ on $\left(-\infty,\frac{33 - 6 p}{8(p+3)}\right)$ and $\rho'(\tau)>0$ on $\left(\frac{33 - 6 p}{8(p+3)},\infty\right)$. Since $p>\frac{545}{261}$, we have
\be
\rho\left(\frac{9 (p-2)}{7-3 p}\right)=\frac{9 (p-2) (261 p-545)}{(7-3 p)^2}>0.\label{dfakljhfalkdsjfalkjshdka}
\ee
Moreover, since $p>\frac{545}{261}>\frac{1}{36} \left(\sqrt{20953}-71\right)$, we have
\bess
\rho'(\tau)&=&8(3+p)\tau+6p-33\\
&>&8(3+p)\cdot\frac{9 (p-2)}{7-3 p}+6p-33\\
&=&\frac{3 \left(18 p^2+71 p-221\right)}{ 7-3 p}\\
&=&3\frac{ \left(p+\frac{1}{36} \left(\sqrt{20953}+71\right)\right)\left(p-\frac{1}{36} \left(\sqrt{20953}-71\right)\right)}{ 7-3 p}>0,\ \ \ {\rm for\ all}\  \tau>\frac{9 (p-2)}{7-3 p}.
\eess
This together with (\ref{lasdkfjlashdfl}) and (\ref{dfakljhfalkdsjfalkjshdka}) enables us to pick $\delta_1=\delta_1(p)$ such that
\be
\frac32(\frac1m-\frac1q)-1=-\frac{\rho(m)}{4 m (m (p+3)+3 (p-2))}<0,\ \ \ \ {\rm for\ all}\ m>\frac{9 (p-2)}{7-3 p}-\delta_1(p).\label{eq-m-range-lem6.u}
\ee

Suppose (\ref{ldkfmaldfnlajdfnad}) is established for some $\varepsilon\in(0,1)$ and $K>0$. An application of the variation-of-constants representation of $\zzh u$ enables us to obtain
\be
\|\zzh u\zbl\|_\lp q\leq\|\banqun t u_0\|_\lp q+\int^t_0\|\banqun{(t-s)}\mathcal{P}[\zzh n\zbls\nabla\phi]\|_\lp q ds,\ \ \ \ t>0.
\ee
Using the regularization properties of the Dirichlet Stokes semigroup $(e^{tA})_{t\geq0}$ (\cite[p.201]{Giga-JDE-1986}), we can find $C_1>0$, $C_2>0$, $C_3>0$ and $\lambda>0$ such that
\be
\|\banqun t u_0\|_\lp q\leq C_1\|u_0\|_{\lp q}\sjj\label{dkfjhakjndfakusdfods}
\ee
and
\bes\label{slfknjdkasjnfiauhdf}
&&\int^t_0\|\banqun{(t-s)}\mathcal{P}[\zzh n\zbls\nabla\phi]\|_\lp q ds\nonumber\\
&\leq& C_2\int^t_0\left(1+(t-s)^{-\frac32(\frac1m-\frac1q)}\right)e^{-\lambda(t-s)}\|\mathcal{P}[\zzh n\zbls\nabla\phi]\|_\lp m ds\sjj.
\ees
From (\ref{ldkfmaldfnlajdfnad}), the boundedness of $\nabla\phi$ on $\Omega$ and the continuity of the Helmholtz projection acting as an operator in $L^p(\Omega;\daR^3)$ (\cite{Fujiwara-1997}), we see that there exist some $C_3>0$ such that
\be
\|\mathcal{P}[\zzh n\zbls\nabla\phi]\|_\lp m\leq C_3\|\zzh n\zbls\|_\lp m \leq C_3K^\frac1m,\ \ \ \ {\rm for\ all}\ s>0.\nonumber\\
\ee
This together with (\ref{slfknjdkasjnfiauhdf}) shows that

\be
\int^t_0\|\banqun{(t-s)}\mathcal{P}[\zzh n\zbls\nabla\phi]\|_\lp q ds\leq C_2C_3C_4K^{\frac1m}\sjj\label{ksudkjafiauhodsandka}
\ee
with $C_4=\int^\infty_0\left(1+\sigma^{-\frac32(\frac1m-\frac1q)}\right)e^{-\lambda\sigma}d\sigma$ being finite thanks to (\ref{eq-m-range-lem6.u}).
\end{proof}
Using the pair of results (\ref{n-m-result-lem4.1}) and (\ref{nabla-n-m-2-result-lem4.1}) in Lemma \ref{lem4.1}, we can obtain the following space-time integrability of $\zzh n$.
\begin{lem}\label{lem6.3}
Let $p\geq2$. Then for all $m\geq1$ and any $K>0$, if
\be
\int_\Omega n_\varepsilon^m(\cdot,t)\leq K,\ \ \ \ {for\ all}\ t\geq0\label{adsfvadfagDSADA}
\ee
and
\be
\int^{t+1}_t\int_\Omega\left|\nabla n_\varepsilon^{\frac{m-2}p+1}\right|^p\leq K,\ \ \ \ {for\ all}\ t\geq0,\label{kldsafjlndaiksjdlkaj}
\ee
then there exist $C(m,p,K)>0$ such that
\be
\int^{t+1}_t\int_\Omega n_\varepsilon^{\frac{m p}{3}+m+p-2}\leq C(m,p,K),\ \ \ \ {for\ all}\ t\geq0.\label{n-result-lem6.3}
\ee
\end{lem}

\begin{proof} Let $\alpha=\frac{m p}{3}+m+p-2$ and $\widetilde{m}=\frac{m-2}p+1$, then
\bess
\frac{\widetilde{m}}\alpha-\left(\frac1p-\frac13\right)=\frac{(m+3) p-6}{3 (m (p+3)+3 (p-2))}>0,
\eess
and hence $W^{1,p}(\Omega)\hookrightarrow L^{\frac\alpha{\widetilde{m}}}(\Omega)\hookrightarrow L^{\frac m{\widetilde{m}}}(\Omega)$. Define
\bess
\theta=\frac{3 (m+p-2)}{m (p+3)+3 (p-2)}\in(0,1),
\eess
then $\theta$ satisfies
\bess
{\frac{\widetilde{m}}\alpha}=\theta(\frac1p-\frac13)+(1-\theta)\frac{\widetilde{m}}m\ \ \ {\rm and}\ \ \ \frac\alpha{\widetilde{m}}\theta=p.
\eess
Using (\ref{adsfvadfagDSADA}), (\ref{kldsafjlndaiksjdlkaj}) and the Gagliardo-Nirenberg inequality, we obtain $C_1>0$ such that
\bess
\jfo \zzh n^\alpha=\left\|\zzh n^{\widetilde{m}}\right\|^{\frac\alpha{\widetilde{m}}}_{\frac\alpha{\widetilde{m}}}&\leq& C_1\left\|\nabla \zzh n^{\widetilde{m}}\right\|^{\frac\alpha{\widetilde{m}}\theta}_p\left\|\zzh n^{\widetilde{m}}\right\|^{\frac\alpha{\widetilde{m}}(1-\theta)}_{\frac m{\widetilde{m}}}+C\left\|\zzh n^{\widetilde{m}}\right\|^{\frac\alpha{\widetilde{m}}}_{\frac m{\widetilde{m}}}\\
&\leq& C_1K^{\frac{\alpha(1-\theta)}m}\left\|\nabla \zzh n^{\widetilde{m}}\right\|^p_p+C_1K^{\frac\alpha m}\\
&\leq& C_1K^{\frac{p}{3}}\left\|\nabla \zzh n^{\widetilde{m}}\right\|^p_p+C_1K^{1+\frac p3+\frac{p-2}m}\sjj,
\eess
which further implies
\bess
\jft \zzh n^\alpha&\leq& C_1K^{\frac{p}{3}}\int^{t+1}_t\int_\Omega\left|\nabla n_\varepsilon^{\frac{m-2}p+1}\right|^p+C_1K^{1+\frac p3+\frac{p-2}m}\\
&\leq& C_1K^{1+{\frac{p}{3}}}+C_1K^{1+\frac p3+\frac{p-2}m},\ \ \ \ {\rm for\ all}\ t\geq0.
\eess
\end{proof}
With Lemma \ref{lem6.u} and Lemma \ref{lem6.3} at hand, we can improve the estimate for $\nabla\zzh c$.
\begin{lem}\label{lem6.c}
Let $p\in(\frac{545}{261},\frac73)$ and let $\delta_1(p)>0$ be as in Lemma \ref{lem6.u}. Then for all $m>\frac{9 (p-2)}{7-3 p}-\delta_1(p)$ and any $K>0$, if for any $\eps\in(0,1)$ there hold
\be
\jfo \zzh n^m(\cdot,t)\leq K\sjl,
\ee
and
\be
\int^{t+1}_t\int_\Omega|\nabla\zzh n^{\frac{m-2}p+1}|^p\leq K\sjl,
\ee
then one can find $C(m,p,K)>0$ with the property that
\be
\int^{t+1}_t\int_\Omega|\nabla\zzh c|^{2(\zs)}\leq C(m,p,K)\sjl.
\ee
\end{lem}
\begin{proof}
Similar as the proof of Lemma 6.3 in \cite{Winkler-JDE-2018}( we only need to set $q=\zs$ here), then we can get the desired result.
\end{proof}
\section{$L^\infty$ bounds for $\zzh n$ when $p>\frac{23}{11}$}
According to Lemma \ref{lem6.c}, we can improve the estimate of Lemma \ref{lem4.1}.
\begin{lem}\label{lem7.sec-ire}
Let $p\in(\frac{545}{261},\frac73)$, $m_*>\frac{9 (p-2)}{7-3 p}-\delta_1(p)$ with $\delta_1(p)>0$ taken from Lemma \ref{lem6.u}. 
Then for all $m>1$ fulfilling
\be
m\leq \frac{1}{9} \left\{2(p+3)(p-1)m_*^2+3(4p^2-5p-8)m_*+9(2p-5)(p-2)\right\},\label{m-constant-condition-lem6.1}
\ee
and any $K>0$, if for any $\eps\in(0,1)$ there hold
\be
\int_\Omega n_\varepsilon^{m_*}(\cdot,t)\leq K,\ \ \ \ {for\ all}\ t\geq0\label{n-m-0-condition-lem7.1}
\ee
and
\be
\int^{t+1}_t\int_\Omega|\nabla n_\varepsilon^{\frac{m_*-2}p+1}|^p\leq K,\ \ \ \ {for\ all}\ t\geq0,\label{nabla-c-4-condition-lem7.1}
\ee
then one can pick $C(m_*,m,p,K)>0$ such that
\be
\int_\Omega n_\varepsilon^m(\cdot,t)\leq C(m_*,m,p,K),\ \ \ \ {for\ all}\ t\geq0\label{n-m-result-lem7.1}
\ee
and
\be
\int^{t+1}_t\int_\Omega|\nabla n_\varepsilon^{\frac{m-2}p+1}|^p\leq C(m_*,m,p,K),\ \ \ \ {for\ all}\ t\geq0.\label{nabla-n-m-2-result-lem7.1}
\ee
\end{lem}
\begin{proof}
Since $m_*>\frac{9 (p-2)}{7-3 p}-\delta_1(p)$, we can apply Lemma \ref{lem6.c} to $q=\frac{m_*p}3+m_*+p-2$ to see that for some $C_1(m,p,K)>0$ there holds
\be
\int^{t+1}_t\int_\Omega|\nabla\zzh c|^{2q}\leq C_1(m,p,K)\sj.
\ee
Substituting $q=\frac{m_*p}3+m_*+p-2$ into the right-hand side of (\ref{m-constant-condition-lem4.1}), we obtain
\bess
&&\frac{2q(p-1)-p}3 m_*+(2q-1)(p-2)\\
&=&\frac{1}{9} \left\{2(p+3)(p-1)m_*^2+3(4p^2-5p-8)m_*+9(2p-5)(p-2)\right\}.
\eess
Then for any $m>1$ satisfies (\ref{m-constant-condition-lem6.1}), an application of Lemma \ref{lem4.1} shows that both inequalities in (\ref{n-m-result-lem7.1}) and (\ref{nabla-n-m-2-result-lem7.1}) hold for some suitably large $C(m_*,m,p,K)>0$ due to (\ref{m-constant-condition-lem6.1}) and (\ref{n-m-0-condition-lem7.1}).
\end{proof}
In order to find out through its condition (\ref{m-constant-condition-lem6.1}), how far Lemma \ref{lem7.sec-ire} can improve the regularity of $\zzh n$ and $\nabla \zzh n$, let us first demonstrate the following observation, which highlight the role of restriction $p>\frac{23}{11}$ made in Theorem \ref{thm1.1}. 
\begin{lem}\label{lem7.2}
For $p\in(2,\frac73)$, let
\be
\psi(m)=\frac{1}{9} \left\{2(p+3)(p-1)m^2+3(4p^2-5p-8)m+9(2p-5)(p-2)\right\},\ \ \ \ p\in\mathbb{R}.\label{sdfakhfdkla}
\ee
Then
\be
\psi(\frac{9 (p-2)}{7-3 p})>\frac{9 (p-2)}{7-3 p}\ \ \ if\ and\ only\ if\ \ p>\frac{23}{11},\label{eq1-lem7.2}
\ee
and there exist $\delta_2(p)>0$ and $\Gamma>1$ such that
\be
\psi(m)>m\Gamma,\ \ \ for\ all\ \ m>\frac{9 (p-2)}{7-3 p}-\delta_2(p).\label{eq2-lem7.2}
\ee
\end{lem}
\begin{proof}
Since $p>2$ and
\bess
\psi(\frac{9 (p-2)}{7-3 p})-\frac{9 (p-2)}{7-3 p}=\frac{16 (p-2) (11 p-23)}{(7-3 p)^2}>0,
\eess
we directly obtain (\ref{eq1-lem7.2}). To prove (\ref{eq2-lem7.2}), we let
\be
\tilde{\psi}(m)=\frac{\psi(m)}m,\ \ \ \ {\rm for\ all\ }m>0,\label{sflashflkhwbf}
\ee
this together with (\ref{eq1-lem7.2}) yields that $C_1=\tilde{\psi}(\frac{9 (p-2)}{7-3 p})-1$ is positive. The continuity of $\tilde{\psi}(m)$ allows us to pick $\delta_2=\delta_2(p)>0$ such that $\frac{9 (p-2)}{7-3 p}-\delta_2(p)>0$ and
\be
\tilde{\psi}(m)\geq\Gamma=1+\frac{C_1}2,\ \ \ \ {\rm for\ all\ }m\in\left(\frac{9 (p-2)}{7-3 p}-\delta_2(p),\frac{9 (p-2)}{7-3 p}\right].\label{kdlfjhkdslflduasrfyhe}
\ee
Using (\ref{sdfakhfdkla}) and (\ref{sflashflkhwbf}), we obtain
\bess
\tilde{\psi}'(m)=\frac{2}{9} (p-1) (p+3)-\frac{(p-2) (2 p-5)}{m^2},\ \ \ \ {\rm for\ all\ }m>0.
\eess
If $p\leq\frac52$, then $\tilde{\psi}'(m)\geq\frac{2}{9} (p-1) (p+3)>0$ for $m\in(0,\infty)$. If $p>\frac52$, then for $m\geq \frac{9 (p-2)}{7-3 p}$, there holds
\bess
\tilde{\psi}'(m)\geq\frac{2}{9} (p-1) (p+3)-(p-2) (2 p-5)\left(\frac{7-3 p}{9 (p-2)}\right)^2=\frac{1}{81} \left(129 p+\frac{1}{p-2}-176\right)>0.
\eess
In both case, we obtain that $\tilde{\psi}'>0$ on $\left[\frac{9 (p-2)}{7-3 p},\infty\right)$ and hence $\tilde{\psi}\geq \Gamma$ on $\left[\frac{9 (p-2)}{7-3 p}-\delta_2(p),\infty\right)$ by (\ref{kdlfjhkdslflduasrfyhe}).
\end{proof}
\begin{lem}\label{jsafdakldjsl}
Let $p\in(\frac{23}{11},\frac73)$. Then for all $m>1$, there exist $C=C(m)>0$ such that
\be
\int_\Omega n_\varepsilon^{m}(\cdot,t)\leq C(m),\ \ \ \ {for\ all}\ t\geq0\label{fsdghjyasdtgfar}
\ee
and
\be
\int^{t+1}_t\int_\Omega|\nabla n_\varepsilon^{\frac{m-2}p+1}|^p\leq C(m),\ \ \ \ {for\ all}\ t\geq0,\label{fgdanartgafgs}
\ee
\end{lem}
\begin{proof} Since $p>\frac{23}{11}>\frac{545}{261}$, we can pick $m_0\in (1,\frac{9 (p-2)}{7-3 p})$ such that
\be
m_0>\frac{9 (p-2)}{7-3 p}-\min\{\delta_1(p),\delta_2(p)\},\label{hkfdjhasikfhkahgfv}
\ee
where $\delta_1(p)$ and $\delta_2(p)$ are given by Lemma \ref{lem7.sec-ire} and Lemma \ref{lem7.2} respectively. Define
\be
m_k=\psi(m_{k-1}),\ \ \ \ k\in\mathbb{N}=\{1,2,3,\cdots\},\label{lkjdsfaoi}
\ee
with $\psi:\daR\rightarrow\daR$ given by (\ref{sdfakhfdkla}) in Lemma \ref{lem7.2}. Combining (\ref{hkfdjhasikfhkahgfv}) with Lemma \ref{lem7.2}, an inductive argument shows that
\be
m_k\geq\Gamma^km_0,\ \ \ \ {\rm for\ all}\ k\in\mathbb{N}\label{faksdhjlkah}
\ee
with $\Gamma>1$ given by Lemma \ref{lem7.2}, whence in particular $m_k\rightarrow\infty$ as $k\rightarrow\infty$.

Due to the boundedness of $\Omega$, it suffices to show that for all $k\in\mathbb{N}$ there exists $C_1(k)>0$ such that for all $\varepsilon\in(0,1)$,
\be
\int_\Omega n_\varepsilon^{m_k}(\cdot,t)\leq C_1(k)\ \ \ {\rm and}\ \ \ \int^{t+1}_t\int_\Omega|\nabla n_\varepsilon^{\frac{m_k-2}p+1}|^p\leq C_1(k)\sj.\label{kdfjahlhjfda}
\ee
This can be proven again by an iterative approach. Indeed, for $k=0$, (\ref{kdfjahlhjfda}) can be derived from Lemma \ref{lemfhdaskh}, since $p>\frac{23}{11}>\frac{25}{12}$ and $m_0\in (1,\frac{9 (p-2)}{7-3 p})$. If (\ref{kdfjahlhjfda}) holds for some $k_0\geq0$ and some $C_1(k_0)>0$, then since (\ref{hkfdjhasikfhkahgfv}) and (\ref{faksdhjlkah}) ensure that $m_k\geq m_0>\frac{9 (p-2)}{7-3 p}-\delta_1(p)$, and again since $p\in(\frac{545}{261},\frac73)$, Lemma \ref{lem7.sec-ire} provides $C_2>0$ such that
\be
\int_\Omega n_\varepsilon^{m}(\cdot,t)\leq C_1(k)\ \ \ {\rm and}\ \ \ \int^{t+1}_t\int_\Omega|\nabla n_\varepsilon^{\frac{m-2}p+1}|^p\leq C_1(k)\sj\nonumber
\ee
with
\be
m=\frac{1}{9} \left\{2(p+3)(p-1)m_{k_0}^2+3(4p^2-5p-8)m_{k_0}+9(2p-5)(p-2)\right\}.\nonumber
\ee
From (\ref{lkjdsfaoi}) we know that $m=\psi(m_{k_0})=m_{k_0+1}$, which means (\ref{kdfjahlhjfda}) also holds for $k=k_0+1$ and thereby completes the proof.
\end{proof}

Using Lemma \ref{jsafdakldjsl} and Lemma \ref{lem6.u}, we can obtain further regularity of $\zzh u$ and $\nabla \zzh c$ through a standard regularization feature of the heat semigroup.
\begin{lem}\label{dklfhjaskflhel}
Let $p\in(\frac{23}{11},\frac73)$ and $m>1$. Then there exist $C(m)>0$ such that for all $\varepsilon\in(0,1)$, there hold
 \begin{eqnarray}
  \jfo |\zzh u(\cdot,t)|^m\leq C(m)\sjl.\label{euitohagjkhdgafiwqpq}
\end{eqnarray}
and
 \begin{eqnarray}
\jfo |\nabla\zzh c(\cdot,t)|^m\leq C(m)\sjl,\label{dkfaghldkjghalkdhj}
\end{eqnarray}
\end{lem}
\begin{proof} (\ref{euitohagjkhdgafiwqpq}) is a direct consequence of a combination of Lemma \ref{jsafdakldjsl} with Lemma \ref{lem6.u}. Based on (\ref{euitohagjkhdgafiwqpq}) and Lemma \ref{jsafdakldjsl}, using the well-known results on gradient regularity in semilinear heat equation (\cite{Horstmann-Winkler-JDE-2005}), we can arrive at (\ref{dkfaghldkjghalkdhj}).
\end{proof}
A Moser-type iterative procedure which was similar to \cite[Lemma 4.1]{first-paper}, shows the boundedness of $\zzh n(\cdot,t)$ in $\lp\infty$ for all $t\geq0$. For readers' convenience, we shall give the detailed proof here.
\begin{lem}\label{nwuqiong}
If $p>\frac{23}{11}$, then there exists $C>0$ such that for all $\varepsilon\in(0,1)$ we have
\be
\|\zzh n(\cdot,t)\|_{L^\infty(\Omega)}\leq C\sjl.
\ee
\end{lem}
\begin{proof} We now recursively define
\be
m_k=2m_{k-1}+2-p,\ \ \ \ k\in\mathbb{N}=\{1,2,3,\cdots\},\label{mk-moser}
\ee
with
\be
m_0>p-2.\label{m0-moser}
\ee
 We note that (\ref{mk-moser}) and (\ref{m0-moser}) ensure that $\{m_k\}_{k\in\mathds{N}}$ is a nonnegative strictly increasing sequence and
\be
m_k\nearrow\infty\ \ \ {\rm as}\ \ \ k\rightarrow\infty, \label{hdkafhaihelf}
\ee
and moreover there holds
\be
c_1 2^k\leq m_k\leq c_2 2^k,\ \ \ {\rm for\ all}\ \ \ k\in\mathbb{N},\label{dengjiamk}
\ee
where $c_1>0$ and $c_2>0$ which, as all constants $C_1$, $C_2$,$\cdots$ appearing below, are independent of $k$. Define
\be
\theta_k=2\cdot\frac{m_k+p-2}{m_k+p'-2}>2,\label{ladhfsliauewhfl}
\ee
then
\be
\theta'_k=\frac{\theta_k}{\theta_k-1}\in(1,2)\ \ \ \ {\rm and}\ \ \ \  \frac1{\theta'_k}>\frac12>\frac{p'}4. \label{thetayipie}
\ee

Our goal is to derive a recursive inequality for
\be
M_k=\sup_{t\in(0,\infty)}\jfo\widehat{\zzh n}^{m_k}(x,t)dx,\ \ \ \ k\in\mathbb{N},\label{fdhaihieufh}
\ee
where $\widehat{\zzh n}(x,t)=\max\{\zzh n(x,t),1\}$ for $x\in\overline{\Omega}$ and $t\in [0,\infty)$. To this end, we may test the second equation of (\ref{model-approximate}) by $m_k\zzh n^{m_k-1}$ and employ Young's inequality to obtain for $k\geq1$
\bess
&&\frac d{dt}\jfo\zzh n^{m_k}+m_k(m_k-1)\jfo\zzh n^{m_k-2}\left(|\nabla\zzh n^2|+\varepsilon\right)^{\frac{p-2}2}|\nabla\zzh n|^2\nonumber\\
&=&m_k(m_k-1)\jfo\zzh n^{m_k-1} F'_\varepsilon(\zzh n)\chi(\zzh c)\nabla\zzh c\cdot\nabla\zzh n\nonumber\\
&\leq&C_1m_k(m_k-1)\jfo\zzh n^{m_k-1} |\nabla\zzh c|\cdot|\nabla\zzh n|\nonumber\\
&=&C_1m_k(m_k-1)\jfo\zzh n^{\frac{m_k-2}p}|\nabla\zzh n|\cdot\zzh n^{\frac{2-m_k}p+m_k-1}|\nabla\zzh c|\nonumber\\
&\leq&C_1m_k(m_k-1)\jfo\left[\frac1{2C_1}\zzh n^{m_k-2}|\nabla\zzh n|^p+\frac{p-1}p\left(\frac p{2C_1}\right)^{-\frac1{p-1}}\cdot\left(\zzh n^{\frac{2-m_k}p+m_k-1}|\nabla\zzh c|\right)^{p'}\right]
\eess
for all $t>0$. This together with the H${\rm{\ddot{o}}}$lder inequality, (\ref{thetayipie}) and Lemma \ref{dklfhjaskflhel} shows that there exists some $C_{2}>0$ and $C_{3}>0$ such that
\bes
&&\frac d{dt}\jfo\zzh n^{m_k}+\frac{m_k(m_k-1)}2\left(\frac p{m_k+p-2}\right)^p\jfo|\nabla\zzh n^{\frac{m_k-2}p+1}|^p\nonumber\\
&=&\frac d{dt}\jfo\zzh n^{m_k}+\frac{m_k(m_k-1)}2\jfo\zzh n^{m_k-2}|\nabla\zzh n|^p\nonumber\\
&\leq&C_2m_k(m_k-1)\jfo\left(\zzh n^{\frac{2-m_k}p+m_k-1}|\nabla\zzh c|\right)^{p'}\nonumber\\
&=&C_2m_k(m_k-1)\jfo\zzh n^{m_k-2+p'}|\nabla\zzh c|^{p'}\nonumber\\
&\leq&C_2m_k(m_k-1)|\Omega|^{\frac1{\theta'_k}-\frac{p'}4}\left(\jfo\zzh n^{(m_k-2+p')\theta_k}\right)^{\frac1{\theta_k}}\cdot\left(\jfo|\nabla\zzh c|^4\right)^{\frac{p'}4}\nonumber\\
&\leq&C_3m_k(m_k-1)\left(\jfo\zzh n^{(m_k-2+p')\theta_k}\right)^{\frac1{\theta_k}}\sjj.\label{dfilusahoiefh}
\ees

We first estimate the right-hand side of (\ref{dfilusahoiefh}). Noting that $\frac{m_k+p-2}2=m_{k-1}$ and ${\frac{(m_k-2+p')\theta_k}{m_k+p-2}}=2$ by (\ref{mk-moser}) and (\ref{ladhfsliauewhfl}) respectively. This together with the Gagliardo-Nirenberg inequality enables us to find $C_4>0$ such that
\bes\label{eq3-moser}
C_3\left(\jfo\zzh n^{(m_k-2+p')\theta_k}\right)^{\frac1{\theta_k}}
&=&C_3\|\zzh n^{\frac{m_k+p-2}p}\|^{\frac {2p}{\theta_k}}_{2p}\nonumber\\
&\leq&C_4\|\nabla\zzh n^{\frac{m_k+p-2}p}\|_p^{\frac{{2p}a}{\theta_k}}\cdot\|\zzh n^{\frac{m_k+p-2}p}\|_{\frac p2}^{\frac{{2p}(1-a)}{\theta_k}}+C_4\|\zzh n^{\frac{m_k+p-2}p}\|_{\frac p2}^{\frac{{2p}}{\theta_k}}\nonumber\\
&=&C_4\|\nabla\zzh n^{\frac{m_k+p-2}p}\|_p^{\frac{2pa}{\theta_k}}\cdot\|\zzh n^{m_{k-1}}\|_1^{\frac{4(1-a)}{\theta_k}}+C_4\|\zzh n^{m_{k-1}}\|_1^{\frac4{\theta_k}},
\ees
for all $t>0$, with
\be
a=\frac{\frac 6p-\frac 3{2p}}{1-\frac 3p+\frac 6p}=\frac{9}{2 p+6}\in(0,1).\nonumber
\ee
In the view of (\ref{fdhaihieufh}), we can apply Young's inequality to (\ref{eq3-moser}) to obtain
\bes
&&C_3\left(\jfo\zzh n^{(m_k-2+p')\theta_k}\right)^{\frac1{\theta_k}}\nonumber\\
&\leq&C_4M_{k-1}^{\frac{4(1-a)}{\theta_k}}\left(\jfo|\nabla\zzh n^{\frac{m_k+p-2}p}|^p\right)^{\frac{{2}a}{\theta_k}}+C_4M_{k-1}^{\frac4{\theta_k}}\nonumber\\
&\leq&C_4\left( \eta\jfo|\nabla\zzh n^{\frac{m_k+p-2}p}|^p+C_{\eta}
\left(M_{k-1}^{\frac{4(1-a)}{\theta_k}}\right)^{\frac{\theta_k}{\theta_k-{2}a}}\right)+C_4M_{k-1}^{\frac4{\theta_k}}\sjj\label{eq4-moser}
\ees
where $\eta=\frac1{4C_4}\left(\frac p{m_k+p-2}\right)^p>0$ and
\be
C_{\eta}=\left(\eta\frac{\theta_k}{2a}\right)^{-\frac1{\frac{\theta_k}{2a}-1}}\cdot\frac{\frac{\theta_k}{2a}-1}{\frac{\theta_k}{2a}}
=\frac{\theta_k-2a}{\theta_k}\left(\frac{\theta_k}{2a}\eta\right)^{-\frac{2a}{\theta_k-2a}}.\nonumber
\ee
Letting
\be
\tilde{b}=2^{\frac{ap}{1-a}}>1,\label{dfjailifhlu}
\ee
 then by (\ref{dengjiamk}) we have
\bes
C_\eta&\leq&\eta^{-\frac{2a}{\theta_k-2a}}\nonumber\\
&=& \left(4C_4\right)^{\frac{2a}{\theta_k-2a}}\left(\frac {m_k+p-2}p\right)^{\frac{2a}{\theta_k-2a}p}\nonumber\\
&\leq& \left(4C_4\right)^{\frac{2a}{\theta_k-2a}}(m_k^{\frac{2a}{\theta_k-2a}p}+1)\nonumber\\
&\leq&\left(4C_4\right)^{\frac a{1-a}}2m_k^{\frac a{1-a}p}\nonumber\\
&\leq&C_5 \tilde{b}^k\label{ldfkhadsfiauehlo}
\ees
since $\frac a{1-a}>0$ and $\theta_k>2$. A combination of (\ref{eq4-moser})-(\ref{ldfkhadsfiauehlo}) shows that
\bes
&&C_3\left(\jfo\zzh n^{(m_k-2+p')\theta_k}\right)^{\frac1{\theta_k}}\nonumber\\
&\leq&\frac1{4}\left(\frac p{m_k+p-2}\right)^p\jfo|\nabla\zzh n^{\frac{m_k+p-2}p}|^p+C_4C_5 \tilde{b}^k
M_{k-1}^{\frac{4(1-a)}{\theta_k-2a}}+C_4M_{k-1}^{\frac4{\theta_k}}\sjj\label{ydjkdasfgaruj}
\ees
this together with (\ref{ydjkdasfgaruj}) yields
\bes\label{eq5-moser}
&&\frac d{dt}\jfo\zzh n^{m_k}+\frac{m_k(m_k-1)}4\left(\frac p{m_k+p-2}\right)^p\jfo|\nabla\zzh n^{\frac{m_k-2}p+1}|^p\nonumber\\
&\leq&C_4C_{\eta}m_k(m_k-1)\left(M_{k-1}^{\frac{4(1-a)}{\theta_k}}\right)^{\frac{\theta_k}{\theta_k-2a}}+C_4m_k(m_k-1)M_{k-1}^{\frac4{\theta_k}}\nonumber\\
&\leq&C_4C_5m_k(m_k-1)\tilde{b}^kM_{k-1}^{\frac{4(1-a)}{\theta_k-2a}}+C_4m_k(m_k-1)M_{k-1}^{\frac4{\theta_k}}\sjj.
\ees

We can generate a linear absorption term in (\ref{eq5-moser}) again by the routine method. Indeed, we can use the Gagliardo-Nirenberg inequality to estimate
\bes\label{eq6-moser}
\jfo\zzh n^{m_k}&=&\|\zzh n^{\frac{m_k+p-2}p}\|_{\frac{pm_k}{m_k+p-2}}^{\frac{pm_k}{m_k+p-2}}\nonumber\\
&\leq&C_6\|\nabla\zzh n^{\frac{m_k+p-2}p}\|_p^{\frac{pm_k}{m_k+p-2}b}\|\zzh n^{\frac{m_k+p-2}p}\|_{\frac p2}^{\frac{pm_k}{m_k+p-2}(1-b)}+C_6\|\zzh n^{\frac{m_k+p-2}p}\|_{\frac p2}^{\frac{pm_k}{m_k+p-2}}\nonumber\\
&=&C_6\|\nabla\zzh n^{\frac{m_k+p-2}p}\|_p^{\frac{pm_k}{m_k+p-2}b}\|\zzh n^{m_{k-1}}\|_1^{\frac{pm_k}{s(m_k+p-2)}(1-b)}+C_6\|\zzh n^{m_{k-1}}\|_1^{\frac{pm_k}{s(m_k+p-2)}}\nonumber\\
&\leq&C_6\|\nabla\zzh n^{\frac{m_k+p-2}p}\|_p^{\frac{pm_k}{m_k+p-2}b}M_{k-1}^{\frac{2(1-b)m_k}{m_k+p-2}}+C_6M_{k-1}^{\frac{2m_k}{m_k+p-2}}\sjj
\ees
with
\be
b=\frac{\frac2p-\frac{m_k+p-2}{pm_k}}{\frac1n-\frac1p+\frac2p}=\frac{\frac2p\left(1-\frac{m_{k-1}}{m_k}\right)}{\frac1n+\frac1p}\in(0,1),\nonumber
\ee
since $1-\frac{m_{k-1}}{m_k}>0$ and $\frac{m_k+p-2}{pm_k}-\frac1p>0$. Using Young's inequality, we can derive from (\ref{eq6-moser}) that
\bes\label{eq8-moser}
\jfo\zzh n^{m_k}&\leq&C_6\|\nabla\zzh n^{\frac{m_k+p-2}p}\|_p^p+C_6M_{k-1}^{\frac{2m_k}{m_k+\frac{p-2}{1-b}}}+C_6M_{k-1}^{\frac{2m_k}{m_k+p-2}}\nonumber\\
&\leq&C_6\|\nabla\zzh n^{\frac{m_k+p-2}p}\|_p^p+2C_6M_{k-1}^{\frac{2m_k}{m_k+p-2}}\sjj
\ees
this together with (\ref{eq5-moser}) yields
\bes\label{eq9-moser}
&&\frac d{dt}\jfo\zzh n^{m_k}+\frac{m_k(m_k-1)}{4C_6}\left(\frac p{m_k+p-2}\right)^p\jfo\zzh n^{m_k}\nonumber\\
&\leq&C_4C_5m_k(m_k-1)\tilde{b}^kM_{k-1}^{\frac{4(1-a)}{\theta_k-2a}}+C_4m_k(m_k-1)M_{k-1}^{\frac4{\theta_k}}\nonumber\\
&&+\frac{m_k(m_k-1)}{2}\left(\frac p{m_k+p-2}\right)^pM_{k-1}^{\frac{2m_k}{m_k+p-2}}\sjj.
\ees
Since $\theta_k>2$ and $a\in(0,1)$, we conclude that
\be
\frac{4(1-a)}{\theta_k-2a}=\frac {4}{\theta_k}\cdot\frac{1-a}{1-a\frac{2}{\theta_k}}<\frac {4}{\theta_k}<2.\nonumber
\ee
This enables us to infer from (\ref{eq9-moser}) that
\bess
&&\frac d{dt}\jfo\zzh n^{m_k}+\frac{m_k(m_k-1)}{4C_6}\left(\frac p{m_k+p-2}\right)^p\jfo\zzh n^{m_k}\\
&\leq&2C_4C_5m_k(m_k-1)\tilde{b}^kM_{k-1}^2+\frac{m_k(m_k-1)}{2}\left(\frac p{m_k+p-2}\right)^pM_{k-1}^2\sjj
\eess
which can be rewritten in the form of
\bess
\frac d{dt}\jfo\zzh n^{m_k}&\leq&2C_4C_5m_k(m_k-1)\tilde{b}^kM_{k-1}^2\\
&&+\frac{m_k(m_k-1)}{2}\left(\frac p{m_k+p-2}\right)^p\left(M_{k-1}^2-\frac1{2C_6}\jfo\zzh n^{m_k}\right)\sjj.
\eess
An integration of this ODI shows that
\bes
M_k\leq\max\{\jfo n_{0\varepsilon}^{m_k},2C_6\left(4C_4C_5\tilde{b}^k\left(\frac{m_k+p-2}p\right)^p+1\right)M_{k-1}^2\}.
\ees

Therefore, in the case when $2C_6\left(4C_4C_5\tilde{b}^k\left(\frac{m_k+p-2}p\right)^p+1\right)M_{k-1}^2<\jfo n_{0\varepsilon}^{m_k}$ holds for infinitely many $k\geq1$, we obtain
\bess
\sup_{t\in(0,\infty)}\left(\jfo\zzh n^{m_{k-1}}\right)^{\frac1{m_{k-1}}}\leq\left(\jfo n_{0\varepsilon}^{m_k}\right)^{\frac{1}{2m_{k-1}}}
\eess
for all such $k$, and hence conclude that
$$\|\zzh n(t)\|_\infty\leq\|n_{0\varepsilon}\|_\infty\sjj,$$
because ${\frac{m_k}{2m_{k-1}}}\rightarrow1$ as $k\rightarrow\infty$ according to (\ref{mk-moser}) and (\ref{hdkafhaihelf}).

Conversely, upon enlarging $C_6$ if necessary we may assume that
\bess
M_k&\leq&2C_6\left(4C_4C_5\tilde{b}^k\left(\frac{m_k+p-2}p\right)^p+1\right)M_{k-1}^2,\ \ \ \ {\rm for\ all}\ \ k\geq1.
\eess
Then from  (\ref{dengjiamk}) and the definition of $\tilde{b}$ (\ref{dfjailifhlu}), we see that there exist $C_8=32C_4C_5C_6>0$, and $C_9=c_2^pC_8>0$ such that
\bess
M_k&\leq&2C_6\left(4C_4C_5\tilde{b}^k\left(m_k^p+1\right)+1\right)M_{k-1}^2\\
&\leq&16C_4C_5C_6\tilde{b}^k\left(m_k^p+1\right)M_{k-1}^2\\
&\leq& C_8\tilde{b}^k\left(c_22^k\right)^pM_{k-1}^2\\
&\leq&C_9 \left(\tilde{b}2^p\right)^kM_{k-1}^2\\
&=&C_9 \tilde{d}^k M_{k-1}^2,\ \ \ \ {\rm for\ all}\ \ k\geq1,
\eess
where $\tilde{d}=\tilde{b}\cdot 2^p=2^{p\left(\frac a{1-a}+1\right)}=2^{\frac p{1-a}}>1$. By a straight forward induction, this yields
\be
M_k\leq C_9^{\sum^{k-1}_{j=0}2^j}\cdot\tilde{d}^{\sum^{k-1}_{j=0}(k-j)2^j}\cdot M_0^{2^k}\label{iufahieygirfa}
\ee
for all $k\geq1$. Using (\ref{dengjiamk}), the identities ${\sum^{k-1}_{j=0}2^j}\leq 2^k$
and $\sum^{k-1}_{j=0}(k-j)2^j=2^{k+1}$, we have
\be\label{limit2-moser}
\lim_{k\rightarrow\infty}\frac 1{m_k}{2^k}\leq\frac1{c_1},\ \ \ \lim_{k\rightarrow\infty}\frac1{m_k}{\sum^{k-1}_{j=0}2^j}\leq\frac1{c_1}\ \ \ {\rm and\ }\ \ \
\lim_{k\rightarrow\infty}\frac 1{m_k}\sum^{k-1}_{j=0}(k-j)2^j=2^{k+1}\leq\frac2{c_1}.
\ee
Finally, a combination of (\ref{fdhaihieufh}), (\ref{iufahieygirfa}) and (\ref{limit2-moser}) concludes that
\be
\|\zzh n(t)\|_\infty\leq C_9^\frac1{c_1}\cdot\tilde{d}^\frac2{c_1}\|n_{0\varepsilon}\|_\infty\sjj.\nonumber
\ee
\end{proof}
Based on the regularities we have obtained, we can obtain the following estimates.
\begin{lem}{\rm(\cite[Lemma 8.4]{Winkler-JDE-2018})}\label{hdfikaufyha}
There exist $\theta\in(0,1)$ with the property that one can find $C>0$ such that for all $\varepsilon\in(0,1)$,
\be
\|\zzh c\|_{C^\theta(\overline{\Omega}\times[t,t+1])}\leq C\sjl
\ee
and
\be
\|\zzh u\|_{C^\theta(\overline{\Omega}\times[t,t+1])}\leq C\sjl,
\ee
and that for all $\tau>0$ it is possible to choose $C(\tau)>0$ fulfilling
\be
\|\nabla\zzh c\|_{C^\theta(\overline{\Omega}\times[t,t+1])}\leq C(\tau),\ \ \ \ for\ all\ t\geq\tau.
\ee
\end{lem}
Let us now come up with one statement on time regularity of $\zzh n$ in a straightforward way.
\begin{lem}\label{lem7.40}
There exists $C>0$ such that for all $\varepsilon\in(0,1)$ we have
\begin{eqnarray}
\int^T_0\|{\partial_t}{n_\varepsilon}(\cdot,t)\|^{p'}_{(W^{1,p}(\Omega))^*}\leq C(T+1),\ \ \ \ {for\ all}\ T>0,\label{eq1-lem5.2}
\end{eqnarray}
where $p'=\frac p{p-1}$.
\end{lem}
\begin{proof}
For $t>0$ and $\varphi\in C^\infty(\overline{\Omega})$, multiplying the first equation in (\ref{model-approximate}) by $\varphi$,
integrating by parts and using the H${\rm{\ddot{o}}}$lder's inequality, we obtain
\begin{eqnarray}
&&|\int_\Omega{\partial_t}{n_\varepsilon}(\cdot,t)\varphi|\nonumber\\
&=&|-\int_\Omega(|\nabla{n_\varepsilon}|^2+\varepsilon)^\frac{p-2}2\nabla{n_\varepsilon}\cdot\nabla\varphi+\int_\Omega{n_\varepsilon}{F_\varepsilon}'({n_\varepsilon})\nabla{c_\varepsilon}\cdot\nabla\varphi+\int_\Omega{n_\varepsilon}{u_\varepsilon}\cdot\nabla\varphi|\nonumber\\
&\leq&\int_\Omega(|\nabla{n_\varepsilon}|^2+\varepsilon)^\frac{p-2}2|\nabla{n_\varepsilon}|\cdot|\nabla\varphi|+C\|{n_\varepsilon}\nabla{c_\varepsilon}\|_{p'}\|\nabla\varphi\|_{p}+\|{n_\varepsilon}{u_\varepsilon}\|_{p'}\|\nabla\varphi\|_{p}\nonumber\\
&\leq&C_1(\|(|\nabla{n_\varepsilon}|^2+\varepsilon)^\frac{p-1}2\|_{p'}+\|{n_\varepsilon}\nabla{c_\varepsilon}\|_{p'}+\|{n_\varepsilon}{u_\varepsilon}\|_{p'})\|\nabla\varphi\|_{p}\nonumber
\end{eqnarray}
with some $C_1>0$.
Lemma \ref{lem3.2} implies
\begin{eqnarray}
\int^T_0\int_\Omega|\nabla{c_\varepsilon}|^4\leq C_2(T+1),\ \ \ \ {\rm{for\ all}}\ T>0\label{eq-lem5.2-zzh-c-4-norm}
\end{eqnarray}
with some $C_2>0$. Using $\frac16+\frac14<\frac16+\frac3{10}<\frac1{p'}$, (\ref{eq-lem5.2-zzh-c-4-norm}), Lemma \ref{jsafdakldjsl}, Lemma (\ref{nwuqiong}), Lemma \ref{dklfhjaskflhel}, H${\rm{\ddot{o}}}$lder's inequality and Young's inequality, we obtain $C_3>0$, $C_4>0$, and $C_5>0$ such that
\begin{eqnarray}
&&\int^T_0\|{\partial_t}{n_\varepsilon}(\cdot,t)\|^{p'}_{(W^{1,p}(\Omega))^*}\nonumber\\
&\leq& C_3(\int^T_0\int_\Omega(|\nabla{n_\varepsilon}|^2+1)^{\frac{p-1}2{p'}}+\int^T_0\int_\Omega|{n_\varepsilon}\nabla{c_\varepsilon}|^{p'}+\int^T_0\int_\Omega|{n_\varepsilon}{u_\varepsilon}|^{p'})\nonumber\\
&\leq&C_4(\int^T_0\int_\Omega(|\nabla{n_\varepsilon}|^{p}+1)+\int^T_0\int_\Omega|{n_\varepsilon}|^{6}+\int^T_0\int_\Omega|\nabla{c_\varepsilon}|^{4}\nonumber\\
&\ &+\int^T_0\int_\Omega|{n_\varepsilon}|^{6}+\int^T_0\int_\Omega|{u_\varepsilon}|^{\frac{10}3}+|\Omega|T)\nonumber\\
&\leq&C_5(T+1),\ \ \ \ {\rm{for\ all}}\ T>0.\label{eq-lem5.2-estimate-i}
\end{eqnarray}
\end{proof}
\section{Existence of a global bounded weak solution}
In this section we construct global bounded weak solutions for (\ref{CS}), (\ref{eq-i}) and (\ref{eq-bc}).
Based on the estimates we collected in the previous sections, we can get the following.

\begin{lem} \label{dhkfauheliuf}
Let $p>\frac{23}{11}$. Then there exist $(\eps_j)_{j\in\mathbb{N}}\subset(0,1)$ satisfying $\eps_j\searrow0$ as $j\rightarrow\infty$, a null set $N\subset(0,\infty)$ and a triple $(n,c,u)$ of functions $n,c: \Omega\times(0,\infty)\rightarrow[0,\infty)$ and $u: \Omega\times(0,\infty)\rightarrow\daR^3$ such that
\begin{eqnarray}
\zzh n(\cdot,t)\rightarrow n (\cdot,t) &&\quad a.e.\ in\ \Omega\ for\ all\ t\in(0,\infty)\setminus N,\label{rneq1}\\
n_{\varepsilon} \xrsl n              &&\quad \mbox{in}\ L^{\infty}(\Omega\times(0, \infty)),\label{rneq2}\\
n_{\varepsilon}\rightarrow n              &&\quad \mbox{in}\ L_{loc}^{p}(\Omega\times(0, \infty)),\label{rneq3}\\
\nabla n_{\varepsilon}\rightharpoonup \nabla n         && \quad \mbox{in}\ L_{loc}^{p}(\Omega\times(0, \infty)),\label{rneq4}\\
|\nabla {n_\varepsilon}|^{p-2}\nabla n_{\varepsilon}\rightharpoonup |\nabla n|^{p-2}\nabla n &&\quad \mbox{in}\ L_{loc}^{p'}(\Omega\times(0, \infty))\label{rneq5}\\
\zzh c\rightarrow c&&\quad \mbox{in}\ C_{loc}^{0}(\overline{\Omega}\times[0, \infty)),\label{rceq1}\\
\zzh c\xrsl c&&\quad \mbox{in}\ L^\infty((0,\infty);W^{1,s}(\Omega))\ \ \ \ for\ all\ s\in(1,\infty),\label{rceq2}\\
\nabla\zzh c\rightarrow \nabla c&&\quad \mbox{in}\ C_{loc}^{0}(\overline{\Omega}\times[0, \infty)),\label{rceq3}\\
\zzh u\rightarrow u&&\quad \mbox{in}\ C_{loc}^{0}(\overline{\Omega}\times[0, \infty)),\label{rueq1}\\
\zzh u\xrsl u&&\quad \mbox{in}\ L^\infty(\Omega\times(0,\infty))\ \ \ \ and\label{rueq2}\\
\nabla\zzh u \rightharpoonup \nabla u&&\quad \mbox{in}\ L^2_{loc}(\Omega\times[0, \infty))\label{rueq3}
\end{eqnarray}
as $\eps=\eps_j\searrow0$, where $p'=\frac p{p-1}$. Moreover, $(n,c,u)$ forms a global weak solution of (\ref{CS}), (\ref{eq-i}), (\ref{eq-bc}) in the sense of Definition \ref{defnaaaa}, and we have
\be
\jfo n(\cdot,t)=\jfo n_0,\ \ \ \ {\rm for\ all}\ t\in(0,\infty)\setminus N.\label{fldhsadlag}
\ee
\end{lem}
\begin{proof}
From Lemma \ref{lem3.2} and Lemma \ref{nwuqiong} we have that for all $T>0$ there holds
\bes
\int^T_0\jfo|\nabla \zzh n|^p&=&\int^T_0\jfo\left|\frac p{p-1}\zzh n^{1-\frac1p}\nabla\zzh n^{\frac{p-1}p}\right|^p\nonumber\\
&\leq&\left(\frac p{p-1}\right)^p\left\|\zzh n\right\|_{L^\infty((0,T)\times\Omega)}^{p-1}\int^T_0\jfo\left|\nabla\zzh n^{\frac{p-1}p}\right|^p\nonumber\\
&\leq&C(T+1).\label{sdklfhash}
\ees
By Lemma \ref{lem7.40}, we have
\be
\left\|\left(n_{\varepsilon}\right)_t\right\|_{L^{p'}([0, T); (W^{1,p}(\Omega))^{*})}\leq C(T+1),\ \ \ \ {\rm for\ all}\ T>0.\label{slx-space-nt}
\ee
By Lemma \ref{nwuqiong}, (\ref{sdklfhash}) and (\ref{slx-space-nt}), an application of Aubin-Lions lemma yields $(\eps_j)_{j\in\mathbb{N}}\subset(0,1)$ such that $\eps_j\searrow0$ as $j\rightarrow\infty$ and
\begin{eqnarray}
n_{\varepsilon}\rightarrow n              &\quad& \mbox{in}\ L_{loc}^{p}(\Omega\times[0, \infty))\ \mbox{and a.e. in}\ \Omega\times(0, \infty),\label{n-strong-p}
\end{eqnarray}
as $\eps=\eps_j\searrow0$ with some nonnegative function $n$ defined on $\Omega\times(0,\infty)$. So (\ref{n-strong-p}) ensures (\ref{rneq3}) and (\ref{rneq1}) follows by using the Fubini-Tonelli theorem. Moreover, (\ref{sdklfhash}) and (\ref{n-strong-p}) yield (\ref{rneq4}). From Lemma \ref{nwuqiong} and (\ref{n-strong-p}) we know that (\ref{rneq2}) is valid. Now with Lemma \ref{nwuqiong}, (\ref{n-strong-p}) and (\ref{rneq2}) at hand, we can use \cite[Lemma 6.2]{first-paper} to obtain
\begin{eqnarray}
|\nabla {n_\varepsilon}|^{p-2}\nabla n_{\varepsilon}\rightharpoonup |\nabla n|^{p-2}\nabla n &\quad& \mbox{in}\ L_{loc}^{p'}(\overline{\Omega}\times[0, \infty)),\nonumber
\end{eqnarray}
hence we arrive at (\ref{rneq5}), and this is enough to warrant that
\begin{eqnarray}
\left(|\nabla {n_\varepsilon}|^2+\eps\right)^\frac{p-2}2\nabla n_{\varepsilon}\rightharpoonup |\nabla n|^{p-2}\nabla n &\quad& \mbox{in}\ L_{loc}^{p'}(\overline{\Omega}\times[0, \infty)).\label{fhdsiuyheoif}
\end{eqnarray}
Based on the priori estimates provided Lemma \ref{lem3.2}, Lemma \ref{dklfhjaskflhel} and Lemma \ref{hdfikaufyha}, using the Arzela-Ascoli theorem twice we can achieve (\ref{rceq1})-(\ref{rueq3}). We can derive from (\ref{model-approximate}) that$\jfo \zzh n=\jfo n_0$, this together with {rneq2} and Lebesgue's Dominated Convergence Theorem yields (\ref{fldhsadlag}).
\end{proof}

\textbf{Proof of Theorem \ref{thm1.1}.} Theorem \ref{thm1.1} is a direct consequence of a combination of Lemma \ref{dklfhjaskflhel}, Lemma \ref{hdfikaufyha} and Lemma \ref{dhkfauheliuf}.
\section*{\uppercase {Acknowledgments}}
The authors are supported in part by NSF of China (No. 11671079, No. 11701290, No. 11601127 and No. 11171063), and NSF of Jiangsu Provience (No. BK20170896).

\begin{thebibliography}{10}

\bibitem{Bellomo&Bellouquid&Tao&Winkler-M3AS-2015}
{\sc N.~Bellomo, A.~Bellouquid, Y.~Tao, and M.~Winkler}, {\em Toward a
  mathematical theory of {K}eller-{S}egel models of pattern formation in
  biological tissues}, Math. Models Methods Appl. Sci., 25 (2015),
  pp.~1663--1763.

\bibitem{Calvez-CPDE-2012}
{\sc V.~Calvez, L.~Corrias, and M.~A. Ebde}, {\em Blow-up, concentration
  phenomenon and global existence for the {K}eller-{S}egel model in high
  dimension}, Comm. Partial Differential Equations, 37 (2012), pp.~561--584.

\bibitem{Chae-DCDS-2013}
{\sc M.~Chae, K.~Kang, and J.~Lee}, {\em Existence of smooth solutions to
  coupled chemotaxis-fluid equations}, Discrete Contin. Dyn. Syst., 33 (2013),
  pp.~2271--2297.

\bibitem{Chae-CPDE-2014}
\leavevmode\vrule height 2pt depth -1.6pt width 23pt, {\em Global existence and
  temporal decay in {K}eller-{S}egel models coupled to fluid equations}, Comm.
  Partial Differential Equations, 39 (2014), pp.~1205--1235.

\bibitem{Zhi-an-JMAA-2010}
{\sc Y.-S. Choi and Z.-a. Wang}, {\em Prevention of blow-up by fast diffusion
  in chemotaxis}, J. Math. Anal. Appl., 362 (2010), pp.~553--564.

\bibitem{Cong&Liu-KRM-2016}
{\sc W.~Cong and J.-G. Liu}, {\em A degenerate {$p$}-{L}aplacian
  {K}eller-{S}egel model}, Kinet. Relat. Models, 9 (2016), pp.~687--714.

\bibitem{Di&Lorz&Markowich-DCDS-2010}
{\sc M.~Di~Francesco, A.~Lorz, and P.~Markowich}, {\em Chemotaxis-fluid coupled
  model for swimming bacteria with nonlinear diffusion: global existence and
  asymptotic behavior}, Discrete Contin. Dyn. Syst., 28 (2010), pp.~1437--1453.

\bibitem{DOMBROWSKI-PRL-2004}
{\sc C.~Dombrowski, L.~Cisneros, S.~Chatkaew, R.~Goldstein, and J.~Kessler},
  {\em Self-concentration and large-scale coherence in bacterial dynamics},
  Physical Review Letters, 93 (2004), p.~098103.

\bibitem{Duan&Lorz&Markowich-CPDE-2010}
{\sc R.~Duan, A.~Lorz, and P.~Markowich}, {\em Global solutions to the coupled
  chemotaxis-fluid equations}, Comm. Partial Differential Equations, 35 (2010),
  pp.~1635--1673.

\bibitem{Duan&Xiang-IMRN-2014}
{\sc R.~Duan and Z.~Xiang}, {\em A note on global existence for the
  chemotaxis-{S}tokes model with nonlinear diffusion}, Int. Math. Res. Not.
  IMRN,  (2014), pp.~1833--1852.

\bibitem{Fujiwara-1997}
{\sc D.~Fujiwara and H.~Morimoto}, {\em An {$L_{r}$}-theorem of the {H}elmholtz
  decomposition of vector fields}, J. Fac. Sci. Univ. Tokyo Sect. IA Math., 24
  (1977), pp.~685--700.

\bibitem{Giga-JDE-1986}
{\sc Y.~Giga}, {\em Solutions for semilinear parabolic equations in {$L^p$} and
  regularity of weak solutions of the {N}avier-{S}tokes system}, J.
  Differential Equations, 62 (1986), pp.~186--212.

\bibitem{Hillen-JMB-2009}
{\sc T.~Hillen and K.~Painter}, {\em A user's guide to {PDE} models for
  chemotaxis}, J. Math. Biol., 58 (2009), pp.~183--217.

\bibitem{Hittmeir-SIAMJMA-2011}
{\sc S.~Hittmeir and A.~J\"ungel}, {\em Cross diffusion preventing blow-up in
  the two-dimensional {K}eller-{S}egel model}, SIAM J. Math. Anal., 43 (2011),
  pp.~997--1022.

\bibitem{Horstmann-JDMV-2003}
{\sc D.~Horstmann}, {\em From 1970 until present: the {K}eller-{S}egel model in
  chemotaxis and its consequences. {I}}, Jahresber. Deutsch. Math.-Verein., 105
  (2003), pp.~103--165.

\bibitem{Horstmann-JDMV-2004}
\leavevmode\vrule height 2pt depth -1.6pt width 23pt, {\em From 1970 until
  present: the {K}eller-{S}egel model in chemotaxis and its consequences.
  {II}}, Jahresber. Deutsch. Math.-Verein., 106 (2004), pp.~51--69.

\bibitem{Horstmann&Wang-JAM-2001}
{\sc D.~Horstmann and G.~Wang}, {\em Blow-up in a chemotaxis model without
  symmetry assumptions}, European J. Appl. Math., 12 (2001), pp.~159--177.

\bibitem{Horstmann-Winkler-JDE-2005}
{\sc D.~Horstmann and M.~Winkler}, {\em Boundedness vs. blow-up in a chemotaxis
  system}, J. Differential Equations, 215 (2005), pp.~52--107.

\bibitem{Jiang&Wu&Zheng-2014}
{\sc J.~Jiang, H.~Wu, and S.~Zheng}, {\em Global existence and asymptotic
  behavior of solutions to a chemotaxis-fluid system on general bounded
  domains}, Asymptot. Anal., 92 (2015), pp.~249--258.

\bibitem{KS1970}
{\sc E.~Keller and L.~Segel}, {\em Initiation of slime mold aggregation viewed
  as an instability}, Journal of Theoretical Biology, 26 (1970), pp.~399--415.

\bibitem{LiYan-JMAA-2015}
{\sc Y.~Li}, {\em Global bounded solutions and their asymptotic properties
  under small initial data condition in a two-dimensional chemotaxis system for
  two species}, J. Math. Anal. Appl., 429 (2015), pp.~1291--1304.

\bibitem{LiYan&Lankeit-N-2016}
{\sc Y.~Li and J.~Lankeit}, {\em Boundedness in a chemotaxis-haptotaxis model
  with nonlinear diffusion}, Nonlinearity, 29 (2016), pp.~1564--1595.

\bibitem{LiYan&Li-NA-2014}
{\sc Y.~Li and Y.~Li}, {\em Finite-time blow-up in higher dimensional
  fully-parabolic chemotaxis system for two species}, Nonlinear Anal., 109
  (2014), pp.~72--84.

\bibitem{LiYan&Li-NARWA-2016}
\leavevmode\vrule height 2pt depth -1.6pt width 23pt, {\em Blow-up of nonradial
  solutions to attraction--repulsion chemotaxis system in two dimensions},
  Nonlinear Anal. Real World Appl., 30 (2016), pp.~170--183.

\bibitem{LiYan&Li-JDE-2016}
\leavevmode\vrule height 2pt depth -1.6pt width 23pt, {\em Global boundedness
  of solutions for the chemotaxis-{N}avier-{S}tokes system in {$\Bbb{R}^2$}},
  J. Differential Equations, 261 (2016), pp.~6570--6613.

\bibitem{Lions-1980-ARMA}
{\sc P.-L. Lions}, {\em R\'esolution de probl\`emes elliptiques
  quasilin\'eaires}, Arch. Rational Mech. Anal., 74 (1980), pp.~335--353.

\bibitem{Liu&Lorz-AIHP-2011}
{\sc J.-G. Liu and A.~Lorz}, {\em A coupled chemotaxis-fluid model: global
  existence}, Ann. Inst. H. Poincar\'e Anal. Non Lin\'eaire, 28 (2011),
  pp.~643--652.

\bibitem{Lorz-M3AS-2010}
{\sc A.~Lorz}, {\em Coupled chemotaxis fluid model}, Math. Models Methods Appl.
  Sci., 20 (2010), pp.~987--1004.

\bibitem{Lou&Tao&Winkler-SIAMJMA-2014}
{\sc Y.~Lou, Y.~Tao, and M.~Winkler}, {\em Approaching the ideal free
  distribution in two-species competition models with fitness-dependent
  dispersal}, SIAM J. Math. Anal., 46 (2014), pp.~1228--1262.

\bibitem{MR1361006}
{\sc T.~Nagai}, {\em Blow-up of radially symmetric solutions to a chemotaxis
  system}, Adv. Math. Sci. Appl., 5 (1995), pp.~581--601.

\bibitem{MR1887324}
\leavevmode\vrule height 2pt depth -1.6pt width 23pt, {\em Blowup of nonradial
  solutions to parabolic-elliptic systems modeling chemotaxis in
  two-dimensional domains}, J. Inequal. Appl., 6 (2001), pp.~37--55.

\bibitem{Nagai&Senba&Yoshida-FE-1997}
{\sc T.~Nagai, T.~Senba, and K.~Yoshida}, {\em Application of the
  {T}rudinger-{M}oser inequality to a parabolic system of chemotaxis},
  Funkcial. Ekvac., 40 (1997), pp.~411--433.

\bibitem{Sohr-2001-book}
{\sc H.~Sohr}, {\em The {N}avier-{S}tokes equations. An elementary functional
  analytic approach}, Birkh\"auser, Basel, 2001.

\bibitem{Winkler-SIAM-2014}
{\sc C.~Stinner, C.~Surulescu, and M.~Winkler}, {\em Global weak solutions in a
  {PDE}-{ODE} system modeling multiscale cancer cell invasion}, SIAM J. Math.
  Anal., 46 (2014), pp.~1969--2007.

\bibitem{first-paper}
{\sc W.~Tao and Y.~Li}, {\em Global weak solutions for the three-dimensional
  chemotaxis-navier-stokes system with slow $ p $-laplacian diffusion},
  Nonlinear Anal. Real World Appl., 45 (2019), pp.~26--52.

\bibitem{MR3124759}
{\sc Y.~Tao}, {\em Global dynamics in a higher-dimensional repulsion chemotaxis
  model with nonlinear sensitivity}, Discrete Contin. Dyn. Syst. Ser. B, 18
  (2013), pp.~2705--2722.

\bibitem{Tao&Wang-M3AS-2013}
{\sc Y.~Tao and Z.-A. Wang}, {\em Competing effects of attraction vs. repulsion
  in chemotaxis}, Math. Models Methods Appl. Sci., 23 (2013), pp.~1--36.

\bibitem{Tao&Winkler-SIAMJMA-2011}
{\sc Y.~Tao and M.~Winkler}, {\em A chemotaxis-haptotaxis model: the roles of
  nonlinear diffusion and logistic source}, SIAM J. Math. Anal., 43 (2011),
  pp.~685--704.

\bibitem{Tao&Winkler-AIHP-2013}
\leavevmode\vrule height 2pt depth -1.6pt width 23pt, {\em Locally bounded
  global solutions in a three-dimensional chemotaxis-{S}tokes system with
  nonlinear diffusion}, Ann. Inst. H. Poincar\'e Anal. Non Lin\'eaire, 30
  (2013), pp.~157--178.

\bibitem{Tuval2005}
{\sc I.~Tuval, L.~Cisneros, C.~Dombrowski, C.~Wolgemuth, J.~Kessler, and
  R.~Goldstein}, {\em Bacterial swimming and oxygen transport near contact
  lines}, Proceedings of the National Academy of Sciences of the United States
  of America, 102 (2005), pp.~2277--2282.

\bibitem{Winkler-JDE-2010}
{\sc M.~Winkler}, {\em Aggregation vs. global diffusive behavior in the
  higher-dimensional {K}eller-{S}egel model}, J. Differential Equations, 248
  (2010), pp.~2889--2905.

\bibitem{Winkler-CPDE-2012}
\leavevmode\vrule height 2pt depth -1.6pt width 23pt, {\em Global large-data
  solutions in a chemotaxis-({N}avier-){S}tokes system modeling cellular
  swimming in fluid drops}, Comm. Partial Differential Equations, 37 (2012),
  pp.~319--351.

\bibitem{Winkler-JMPA-2013}
\leavevmode\vrule height 2pt depth -1.6pt width 23pt, {\em Finite-time blow-up
  in the higher-dimensional parabolic-parabolic {K}eller-{S}egel system}, J.
  Math. Pures Appl. (9), 100 (2013), pp.~748--767.

\bibitem{Winkler-ARMA-2014}
\leavevmode\vrule height 2pt depth -1.6pt width 23pt, {\em Stabilization in a
  two-dimensional chemotaxis-{N}avier-{S}tokes system}, Arch. Ration. Mech.
  Anal., 211 (2014), pp.~455--487.

\bibitem{Winkler-CVPDE-2015}
\leavevmode\vrule height 2pt depth -1.6pt width 23pt, {\em Boundedness and
  large time behavior in a three-dimensional chemotaxis-{S}tokes system with
  nonlinear diffusion and general sensitivity}, Calc. Var. Partial Differential
  Equations, 54 (2015), pp.~3789--3828.

\bibitem{Winkler-AIHP-2016}
\leavevmode\vrule height 2pt depth -1.6pt width 23pt, {\em Global weak
  solutions in a three-dimensional chemotaxis--{N}avier-{S}tokes system}, Ann.
  Inst. H. Poincar\'e Anal. Non Lin\'eaire, 33 (2016), pp.~1329--1352.

\bibitem{Winkler-2017-TAMS}
{\sc M.~Winkler}, {\em How far do chemotaxis-driven forces influence regularity
  in the {N}avier-{S}tokes system?}, Trans. Amer. Math. Soc., 369 (2017),
  pp.~3067--3125.

\bibitem{Winkler-JDE-2018}
\leavevmode\vrule height 2pt depth -1.6pt width 23pt, {\em Global existence and
  stabilization in a degenerate chemotaxis-{S}tokes system with mildly strong
  diffusion enhancement}, J. Differential Equations, 264 (2018),
  pp.~6109--6151.

\bibitem{ZhangQian-NARWA-2014}
{\sc Q.~Zhang}, {\em Local well-posedness for the chemotaxis-{N}avier-{S}tokes
  equations in {B}esov spaces}, Nonlinear Anal. Real World Appl., 17 (2014),
  pp.~89--100.

\bibitem{Zhang&Li-JMAA-2014}
{\sc Q.~Zhang and Y.~Li}, {\em Global existence and asymptotic properties of
  the solution to a two-species chemotaxis system}, J. Math. Anal. Appl., 418
  (2014), pp.~47--63.

\bibitem{Zhang&Li-ZAMP-2015-B}
\leavevmode\vrule height 2pt depth -1.6pt width 23pt, {\em Boundedness in a
  quasilinear fully parabolic {K}eller-{S}egel system with logistic source}, Z.
  Angew. Math. Phys., 66 (2015), pp.~2473--2484.

\bibitem{Zhang&Li-DCDS-2015}
\leavevmode\vrule height 2pt depth -1.6pt width 23pt, {\em Convergence rates of
  solutions for a two-dimensional chemotaxis-{N}avier-{S}tokes system},
  Discrete Contin. Dyn. Syst. Ser. B, 20 (2015), pp.~2751--2759.

\bibitem{Zhang&Li-JDE-2015}
\leavevmode\vrule height 2pt depth -1.6pt width 23pt, {\em Global weak
  solutions for the three-dimensional chemotaxis-{N}avier-{S}tokes system with
  nonlinear diffusion}, J. Differential Equations, 259 (2015), pp.~3730--3754.

\bibitem{Zhang&Zheng-SIAM-2014}
{\sc Q.~Zhang and X.~Zheng}, {\em Global well-posedness for the two-dimensional
  incompressible chemotaxis-{N}avier-{S}tokes equations}, SIAM J. Math. Anal.,
  46 (2014), pp.~3078--3105.

\end{thebibliography}

\end{document}